\documentclass{article}
\usepackage[utf8]{inputenc}  
\usepackage[T1]{fontenc}     
\usepackage{textcomp}        
\usepackage{graphicx}
\usepackage{cite}
\usepackage{}
\usepackage{amsfonts}
\usepackage{amsmath, amssymb, amsthm, amsfonts, cases}
\usepackage{empheq}
\usepackage{mathrsfs}
\usepackage[utf8]{inputenc}  
\usepackage[T1]{fontenc}     
\usepackage{amsmath}
\usepackage{url}
\usepackage{authblk}
\usepackage[usenames]{color}
\usepackage{geometry}
\usepackage{graphicx}
\usepackage{manfnt}
\usepackage{epsfig}
\usepackage{indentfirst}
\usepackage[numbers]{natbib}
\usepackage[
colorlinks=true,
linkcolor=black,
citecolor=black
]{hyperref}
\allowdisplaybreaks[4]
\theoremstyle{plain}
\newtheorem{thm}{Theorem}[section]

\newtheorem{pro}[thm]{Problem}
\newtheorem{lem}[thm]{Lemma}

\theoremstyle{definition}
\newtheorem{defn}[thm]{Definition}
\newtheorem{ass}{Assumption}[section]
\newtheorem{rmk}[thm]{Remark}

\makeatletter\@addtoreset{equation}{section} \makeatother

\begin{document}

	\title{\(H_2/H_\infty\) Control for Continuous-Time Mean-Field Stochastic Systems with Affine Terms}

	\date{}

	\author[a]{Xuling Fang}
	\author[b]{Jun Moon}
	\author[a]{Maoning Tang}
	\author[a]{Qingxin Meng\footnote{Corresponding author.
			\authorcr
			\indent E-mail address:Xuling\_Fang@163.com(X. Fang),
junmoon@hanyang.ac.kr, tmorning@zjhu.edu.cn(M.Tang),  mqx@zjhu.edu.cn (Q.Meng) }}

	\affil[a]{\small{Department of Mathematical Sciences, Huzhou University, Zhejiang 313000,PR  China}}
	
	\affil[b]{\small{Department of Electrical
Engineering, Hanyang University,
Seoul, 04763, South Korea}}

	\maketitle
    \begin{abstract}
        This paper discusses the \( H_2/H_{\infty} \) control problem for continuous-time mean-field linear stochastic systems with affine terms over a finite horizon.
		We employ the Mean-Field Stochastic Bounded Real Lemma (MF-SBRL), which provides the necessary and sufficient conditions to ensure that the \( H_{\infty} \) norm of system perturbations remains below a certain level. By utilizing the Mean-Field Forward-Backward Stochastic Differential Equations (MF-FBSDE), we establish the equivalence conditions for open-loop \( H_2/H_{\infty} \) control strategies. Furthermore, the paper demonstrates that the control problem is solvable under closed-loop conditions if solutions exist for four coupled Difference Riccati Equations (CDREs), two sets of backward stochastic differential equations (BSDEs) and ordinary equations (ODEs). The state-feedback gains for the control strategy can be derived from these solutions, thereby linking the feasibility of open-loop and closed-loop solutions.
    \end{abstract}


	\maketitle
	
	\section{Introduction}
	From its inception, \(H_2/H_\infty\)
	control theory has been a cornerstone of modern control theory, attracting substantial attention and rigorous scholarly investigation. Its distinctive qualities render it highly promising in the pursuit of an optimal equilibrium between system performance and robustness. In particular, within the domain of stochastic control systems, the	\(H_2/H_\infty\) control challenge has risen to prominence as a key research focus. This theory integrates two distinct control approaches: 	\(H_2\) control, which utilizes the 	\(H_2\) norm of the system as a measure of performance, primarily addressing the system's steady-state performance and dynamic responsiveness; and \(H_\infty\) control, which centers around ensuring the robust stability of the system and enhancing its anti-interference capabilities based on the \(H_\infty\) norm. The \(H_2/H_\infty\)
	control framework aims to achieve a balanced compromise between these two paradigms, ensuring that the system meets its performance requirements while maintaining robustness.
	
	The \(H_2/H_\infty\) mixed control strategy ensures robust steady-state performance in the presence of uncertain disturbances, making it a widely adopted approach in engineering to tackle modern control system challenges. In the domain of stochastic system control, research into \(H_2/H_\infty\) control theory has been increasingly in-depth. Recently, scholars have devised tailored \(H_2/H_\infty\) control methodologies for various stochastic systems, such as Ito-type, discrete-time, and Markov jump systems. These methodologies not only expand the theoretical framework of \(H_2/H_\infty\) control but also offer effective tools for designing practical system controls. Over the years, various approaches have been introduced, encompassing the linear matrix inequality (LMI)(\cite{lin2006h2}, \cite{scherer1997multiobjective}) technique, the convex optimization method (\cite{haddad2022mixed},\cite{rotea1991mixed}), the Nash equilibrium strategy (\cite{limebeer1994nash},\cite{zhang2017stochastic}), and the Stackelberg strategy (\cite{jungers2008stackelberg},\cite{xie2021mixed}). Within this array of approaches, the Nash equilibrium strategy stands out as a traditional method for tackling \(H_2/H_\infty\) mixed control problem. The strategy encompasses two players: one dedicated to minimizing the \(H_2\) norm and the other concerned with the most severe disturbances as characterized by the \(H_\infty\) norm. Early contributions by Limebeer and Anderson \cite{limebeer1994nash} converted the deterministic hybrid \(H_2/H_\infty\) control challenge into finding a Nash equilibrium by proposing two performance metrics: one associated with the specified \(H_\infty\) robustness limit and the other indicating \(H_2\) optimality. Chen and Zhang \cite{chen2004stochastic} broadened these findings to random scenarios with noise that depends on the system state. Following this, the Nash-inspired \(H_2/H_\infty\) control doctrine has been expanded to encompass a wide range of areas, including both continuous and discrete-time frameworks, linear and nonlinear systems, and stochastic It\^{o} and jump systems on finite or infinite horizon (\cite{wang2023stochastic},\cite{yueying2017infinite}-\cite{zhang2017stochastic}).
	
	In \cite{yueying2017infinite}, Liu, Hou and Bai investigated the infinite horizon \(H_2/H_\infty\) optimal control problem for discrete-time infinite Markov jump systems with \((x, u, v)\)-dependent noise, proposing control strategies to balance performance and robustness.
	Furthermore, in \cite{wang2023stochastic}, Wang, Meng, Shen and Peng investigated the stochastic \(H_2/H_\infty\) control problem for mean-field stochastic differential systems with \((x, u, v)\)-correlated noise. By establishing the MF-SBRL, this paper provides the necessary and sufficient conditions for the existence of an \(H_2/H_\infty\) controller and proves that the optimal control input and the worst-case disturbance have a linear feedback form based on the state and its expectation. 	Additionally, in \cite{ma2024open}, Ma , Mou and Daniel addressed the open-loop \(H_2/H_\infty\) control problem for discrete-time mean-field linear stochastic systems. By introducing a cost function and a level of disturbance attenuation, researchers propose a solution that not only integrates the stochastic nature of the system with mean-field theory but also offers new insights into the control design of complex stochastic systems, further enriching the application of \(H_2/H_\infty\) control theory in stochastic systems.
	
	In recent years, the exploration of mean-field stochastic systems has garnered significant attention within the field of control theory. The application of mean-field theory has undergone notable advancements in disciplines such as economics, finance, and engineering technology, particularly in the areas of optimal control and stochastic differential equations. The theoretical foundation of these studies is deeply rooted in Backward Stochastic Differential Equations (BSDEs) and Forward-Backward Stochastic Differential Equations (FBSDEs).
	
	Sun Jingrui in \cite{sun2017mean} mainly discusses the open-loop solvability of the mean-field linear quadratic (LQ) optimal control problem, analyzing the impact of convexity and uniform convexity of the cost functional on the solvability of the problem.
	In \cite{sun2019linear}, Sun Jingrui mainly studies linear-quadratic stochastic two-person nonzero-sum differential games and establishes the existence conditions for open-loop and closed-loop Nash equilibria. The former is related to the solvability of forward-backward stochastic differential equations (FBSDEs), while the latter depends on the solutions to coupled symmetric Riccati differential equations. It provides new theoretical support and methodological guidance for research in related fields.
	
	Research interests in mean-field stochastic differential games (e.g., \cite{absalomhosking2012stochastic}, \cite{bensoussan2013mean}, \cite{lasry2007mean}) and mean-field linear-quadratic control problems, both in continuous-time and discrete-time settings, and under finite-time or infinite-time horizons (e.g., \cite{elliott2013discrete}, \cite{yong2013linear}), have been gaining traction. The research methodology of \cite{yong2013linear} primarily involves using variational methods to explore the linear-quadratic optimal control problems for mean-field stochastic differential equations. It employs decoupling techniques to handle the related coupled systems and ultimately presents the optimal control in the form of feedback control. The main content of \cite{yong2013linear} discusses the open-loop solvability of such problems and its relationship with the solutions of two coupled Riccati equations.
	
Against this backdrop, this paper conducts an in-depth investigation into the stochastic \(H_2/H_\infty\) control problem for continuous-time mean-field stochastic systems with affine terms, using a Nash game-theoretic approach. Compared with existing work (e.g., \cite{wang2023stochastic}), our study presents several distinctive features and novel contributions.

First, the state dynamics considered in this paper are governed by a mean-field stochastic differential equation with affine terms, which provides a more general and realistic modeling framework. Second, we derive necessary and sufficient conditions for the solvability of the open-loop \(H_2/H_\infty\) control problem. Notably, under the mean-field setting with an affine term, open-loop solvability does not necessarily imply closed-loop solvability \textemdash a contrast to the classical case.  We further identify additional structural conditions under which closed-loop solvability can be ensured, thus highlighting the complexity introduced by affine terms in stochastic control design.

The main contributions of this paper are summarized as follows:
\begin{itemize}
  \item \textbf{Explicit open-loop Nash equilibrium representation:} We obtain an explicit characterization of the open-loop Nash equilibrium by solving a coupled system of mean-field forward-backward stochastic differential equations (FBSDEs) along with stationarity conditions. This result fills a gap in the existing literature where the role of affine terms in open-loop control strategies has often been overlooked.
\item \textbf{Systematic closed-loop synthesis framework:} We develop a comprehensive framework for closed-loop synthesis involving four CDREs along with associated BSDEs and ODEs. This leads to a linear-plus-affine state-feedback control law that simultaneously ensures \(H_2\) optimality and \(H_\infty\) robustness, providing practical guidance for high-performance control design in mean-field stochastic systems.
 \item \textbf{Clarifying the distinction between open-loop and closed-loop solvability:} We provide a rigorous analysis revealing that, unlike in classical control settings, open-loop solvability does not in general imply closed-loop solvability for the mean-field \(H_2/H_\infty\) problem with affine terms. Additionally, we clarified the sufficient conditions for achieving closed-loop solvability, providing a basis for the design of stable controllers in complex stochastic environments.
\end{itemize}

	The subsequent content of this paper is organized as follows: In Section 2, after introducing relevant notations, we  reformulated the stochastic \(H_2/H_\infty\) control problem and optimal closed-loop strategy for mean-field \(H_2/H_\infty\) control. Section 3 derives the stochastic mean-field bounded real lemma, which occupies a central position in our analysis of the \(H_2/H_\infty\) control problem. Section 4 delves into the open-loop solvability and closed-loop solvability of the \(H_2/H_\infty\) control problem. Section 5 summarizes our findings and discuss possible directions for future research.
	
	\section{Preliminaries}
	\subsection{Basic Notations}
In this article, we consider a complete filtered probability space \((\Omega, \mathcal{F}, \mathbb{F}, \mathbb{P})\), on which a one-dimensional standard Brownian motion \(W(\cdot) = \{W(t): 0 \leq t \leq T\}\) is defined. Here, \(\Omega\) is the sample space, \(\mathcal{F}\) is a \(\sigma\)-algebra of subsets of \(\Omega\), and \(\mathbb{P}\) is a probability measure. The filtration \(\mathbb{F} = \{\mathcal{F}_t\}_{0 \leq t \leq T}\) is the natural filtration generated by \(W(\cdot)\), augmented to satisfy the usual conditions of right-continuity and completeness.

We adopt the following notations throughout the paper:
\begin{itemize}
    \item \(\mathbb{R}^n\): The \(n\)-dimensional Euclidean space with norm \(|x| = \sqrt{\sum_{i=1}^n x_i^2}\) for \(x = (x_1, \dots, x_n) \in \mathbb{R}^n\).

    \item \(\mathbb{R}^{n \times m}\): The space of all real \(n \times m\) matrices.

    \item \(\langle \alpha, \beta \rangle\): The standard inner product in \(\mathbb{R}^n\), for \(\alpha, \beta \in \mathbb{R}^n\).

    \item \(\langle M, N \rangle := \operatorname{tr}(M^\top N)\): The Frobenius inner product in \(\mathbb{R}^{n \times m}\) for matrices \(M, N \in \mathbb{R}^{n \times m}\). Here, \(\operatorname{tr}(A)\) denotes the trace of matrix \(A\).

    \item \(|M| := \sqrt{\operatorname{tr}(M^\top M)}\): The Frobenius norm of a matrix \(M \in \mathbb{R}^{n \times m}\).

    \item \(I\): The identity matrix of appropriate dimension.

    \item \(\mathbb{S}^n\): The space of real symmetric \(n \times n\) matrices.

    \item \(\mathbb{S}_+^n\): The cone of symmetric positive semi-definite \(n \times n\) matrices in \(\mathbb{S}^n\).

   \item \(C([t,T]; \mathbb{R}^n)\): The space of continuous functions \(X: [t,T] \to \mathbb{R}^n\), equipped with the uniform norm
\[
\|X\|_{C([t,T]; \mathbb{R}^n)} := \sup_{s \in [t,T]} |X(s)| < \infty.
\]

    \item \(C([t,T]; \mathbb{R}^{d \times n})\): The space of continuous functions \(X: [t,T] \to \mathbb{R}^{d \times n}\) satisfying
    \[
        \|X\|_{C([t,T]; \mathbb{R}^{d \times n})} := \sup_{s \in [t,T]} |X(s)| < \infty.
    \]

    \item \(L^2_{\mathbb{F}}(t,T; \mathbb{R}^d)\): The space of \(\mathbb{F}\)-progressively measurable processes \(X: [t,T] \times \Omega \to \mathbb{R}^d\) such that
    \[
        \|X\|^2_{L^2_{\mathbb{F}}(t,T; \mathbb{R}^d)} := \mathbb{E}\left[\int_t^T |X(s)|^2 ds \right] < \infty.
    \]

    \item \(L^2(\Omega, \mathcal{F}_t, \mathbb{P}; \mathbb{R}^n)\): The space of square-integrable \(\mathcal{F}_t\)-measurable random variables \(\xi: \Omega \to \mathbb{R}^n\) satisfying \(\mathbb{E}[|\xi|^2] < \infty\).

    \item \(L^2_{\mathbb{F}}(\Omega; C([t,T]; \mathbb{R}^n))\): The space of \(\mathbb{F}\)-adapted continuous processes \(X: [t,T] \times \Omega \to \mathbb{R}^n\) such that
    \[
        \|X\|_{L^2_{\mathbb{F}}(\Omega; C([t,T]; \mathbb{R}^n))}^2 := \mathbb{E}\left[ \sup_{s \in [t,T]} |X(s)|^2 \right] < \infty.
    \]
\end{itemize}

\subsection{Problem Statement}
This paper investigates the mixed \(H_2/H_{\infty}\) control problem for continuous-time mean-field stochastic systems with affine terms. Such systems naturally arise in engineering, economics, and physics, where they model the collective behavior of a large number of interacting agents. The system dynamics are described by the following mean-field stochastic differential equation (SDE), driven by a one-dimensional Brownian motion \(W(\cdot)\), over a finite time horizon \([0,T]\):
\begin{equation}\label{eq:2.1}
\begin{cases}
	dX(s) = \big\{A_{1}(s)X(s) + \bar{A}_{1}(s)\mathbb{E}[X(s)] + B_{1}(s)u(s) + \bar{B}_{1}(s)\mathbb{E}[u(s)] \\
	\qquad\qquad + C_{1}(s)v(s) + \bar{C}_{1}(s)\mathbb{E}[v(s)] + b(s)\big\}ds \\
	\qquad\qquad + \big\{A_{2}(s)X(s) + \bar{A}_{2}(s)\mathbb{E}[X(s)] + B_{2}(s)u(s) + \bar{B}_{2}(s)\mathbb{E}[u(s)] \\
	\qquad\qquad + C_{2}(s)v(s) + \bar{C}_{2}(s)\mathbb{E}[v(s)] + \sigma(s)\big\}dW(s), \\
	z(s) = \begin{pmatrix} Q(s)X(s) \\ N_1(s)u(s) \end{pmatrix}, \quad s \in [t,T], \\
	X(t) = \xi.
\end{cases}
\end{equation}

Here, the state process \(X(\cdot) \in \mathbb{R}^n\), denoted by \(X(\cdot; t, \xi, u(\cdot), v(\cdot))\), evolves under the control input \(u(\cdot) \in L^2_{\mathbb{F}}(t,T; \mathbb{R}^{n_u})\), the exogenous disturbance \(v(\cdot) \in L^2_{\mathbb{F}}(t,T; \mathbb{R}^{n_v})\), and the initial condition \(\xi \in L^2(\Omega, \mathcal{F}_t, \mathbb{P}; \mathbb{R}^n)\). The system incorporates mean-field interactions via the dependence on the expectations \(\mathbb{E}[X(s)]\), \(\mathbb{E}[u(s)]\), and \(\mathbb{E}[v(s)]\) in both the drift and diffusion terms. The  process \(z(s) \in \mathbb{R}^m\) represents the system's performance output.

We impose the following assumptions to guarantee well-posedness.
\begin{ass}\label{ass:3.1}
The coefficient functions \(A_i(\cdot), \bar{A}_i(\cdot), B_i(\cdot), \bar{B}_i(\cdot), C_i(\cdot), \bar{C}_i(\cdot)\) for \(i=1,2\), as well as the matrix-valued functions \(Q(\cdot)\) and \(N_1(\cdot)\), are deterministic and uniformly bounded on \([0,T]\). More precisely, there exists a constant \(M > 0\) such that
\[
|A_i(s)|,\; |\bar{A}_i(s)|,\; |B_i(s)|,\; |\bar{B}_i(s)|,\; |C_i(s)|,\; |\bar{C}_i(s)|,\; |Q(s)|,\; |N_1(s)| \leq M,\qquad \forall s\in [0,T].
\]
\end{ass}
\begin{ass}\label{ass:3.2}
The affine terms \(b(\cdot), \sigma(\cdot)\)  \(\in L^2_{\mathbb{F}}(t,T; \mathbb{R}^n)\).
\end{ass}

\begin{ass}\label{ass:3.3}
The matrix \(N_1(\cdot)\) satisfies the orthonormality condition
\[
N_1(s)^\top N_1(s) = I, \quad \forall s \in [0,T],
\]
where \(I\) denotes the identity matrix of appropriate dimension.
\end{ass}

Under Assumptions~\ref{ass:3.1} and~\ref{ass:3.2}, it follows from Proposition 2.1 in \cite{sun2017mean} that for any admissible input triple \((\xi, u(\cdot), v(\cdot)) \in L^2(\Omega, \mathcal{F}_t; \mathbb{R}^n) \times L^2_{\mathbb{F}}(t,T; \mathbb{R}^{n_u}) \times L^2_{\mathbb{F}}(t,T; \mathbb{R}^{n_v})\), the state equation \eqref{eq:2.1} admits a unique strong solution \(X(\cdot) \in L^2_{\mathbb{F}}(\Omega; C(t,T; \mathbb{R}^n))\), which depends continuously on the initial data and system coefficients. Moreover, the following a priori estimate holds.
	\begin{lem}
	For any \((t, \xi) \in [0,T] \times L^2(\Omega,\mathbb{P},\mathcal{F}_t;\mathbb{R}^n)\), \(u(\cdot) \in L^2_{\mathbb{F}}(t,T; \mathbb{R}^{n_u})\), and \(v(\cdot) \in L^2_{\mathbb{F}}(t,T; \mathbb{R}^{n_v})\), system \eqref{eq:2.1} admits a unique solution \(X(\cdot) \in L^2_{\mathbb{F}}(\Omega;C(t,T;\mathbb{R}^n))\) satisfying
	\[
	\mathbb{E}\left[ \sup_{s \in [t,T]} |X(s)|^2 \right] \leq M \left( \mathbb{E}|\xi|^2 + \mathbb{E} \int_t^T \left( |u(s)|^2 + |v(s)|^2 + |b(s)|^2 + |\sigma(s)|^2 \right) ds \right),
	\]
	where \(M > 0\) is a constant independent of \(t, \xi, u(\cdot), v(\cdot), b(\cdot)\), and \(\sigma(\cdot)\).
\end{lem}

	This paper will consider the synthesis of a closed-loop control strategy for a finite-horizon mean-field \( H_2/H_\infty \) control problem.  The goal is to minimize the expected output energy while guaranteeing a prescribed disturbance attenuation level. We now proceed to formally define the performance criteria and the closed-loop controller design problem.

	\begin{defn}\label{def:1} \textbf{(Mean-Field \(H_2/H_\infty\) Optimal Closed-Loop Strategy)}
Consider the controlled mean-field system \eqref{eq:2.1} on the interval \([t, T]\), with a prescribed disturbance attenuation level \(\gamma > 0\). Define the admissible closed-loop control strategy spaces as:
\[
\mathcal{N}^2[t, T] \triangleq C(t, T ;\, \mathbb{R}^{n_{u}\times n}) \times C(t, T ; \mathbb{R}^{n_{u}\times n}) \times L_{\mathbb{F}}^{2}(t, T ; \mathbb{R}^{n_{u}}),
\]
\[
\mathcal{M}^2[t, T] \triangleq C(t, T ; \mathbb{R}^{n_{v}\times n}) \times C(t, T ; \mathbb{R}^{n_{v}\times n}) \times L_{\mathbb{F}}^{2}(t, T ; \mathbb{R}^{n_{v}}).
\]
We say that the \(H_2/H_\infty\) problem admits an \emph{optimal closed-loop strategy} if there exists a pair of control strategy
\[
\left(U(\cdot), \bar{U}(\cdot), U_0(\cdot); V(\cdot), \bar{V}(\cdot), V_0(\cdot)\right) \in \mathcal{N}^2[t, T] \times \mathcal{M}^2[t, T],
\]
such that the following two conditions hold:
\begin{itemize}
	\item[(i)] (\textbf{\(H_\infty\) Robustness}) Consider the closed-loop system:
	\begin{equation}\label{eq:closed_loop}
		\begin{cases}
			\begin{aligned}
				dX(s) &= \Big[ \left(A_1(s) + B_1(s)U(s)\right)X(s) + \left(\bar{A}_1(s) + B_1(s)\bar{U}(s) + \bar{B}_1(s)\widetilde{U}(s)\right)\mathbb{E}[X(s)] \\
				&\quad + C_1(s)v(s) + \bar{C}_1(s)\mathbb{E}[v(s)] \Big] ds \\
				&\quad + \Big[ \left(A_2(s) + B_2(s)U(s)\right)X(s) + \left(\bar{A}_2(s) + B_2(s)\bar{U}(s) + \bar{B}_2(s)\widetilde{U}(s) \right)\mathbb{E}[X(s)] \\
				&\quad + C_2(s)v(s) + \bar{C}_2(s)\mathbb{E}[v(s)] \Big] dW(s), \\
				z(s) &= \begin{pmatrix} Q(s)X(s) \\ N_1(s)\left(U(s)X(s) + \bar{U}(s)\mathbb{E}[X(s)]\right) \end{pmatrix}, \\
				X(t) &= 0,
			\end{aligned}
		\end{cases}
	\end{equation}
	where \(\widetilde{U}(s)=U(s)+\bar{U}(s)\),
	and define the linear operator \(\mathbf{L}: L_{\mathbb{F}}^2(t, T; \mathbb{R}^{n_v}) \rightarrow L_{\mathbb{F}}^2(t, T; \mathbb{R}^{n_z})\) by \(\mathbf{L}(v(\cdot)) \triangleq z(\cdot)\). The operator norm of \(\|\mathbf{L}\|\) is given by
	\begin{equation}\label{eq:L_norm}
		\|\mathbf{L}\| \triangleq \sup_{\substack{v(\cdot) \neq 0 \\ v \in L_{\mathbb{F}}^2([t, T]; \mathbb{R}^{n_v})}} \frac{\|z\|_{[t,T]}}{\|v\|_{[t,T]}},
	\end{equation}
	which satisfies \(\|\mathbf{L}\| < \gamma\),
	where
	\[
	\|v\|_{[t,T]} \triangleq \left(\mathbb{E} \int_t^T |v(s)|^2 ds\right)^{1/2}, \quad \|z\|_{[t,T]} \triangleq \left(\mathbb{E} \int_t^T \left(|Q(s)X(s)|^2 + |U(s)X(s) + \bar{U}(s)\mathbb{E}[X(s)]|^2\right) ds\right)^{1/2}.
	\]
	\item[(ii)] (\textbf{\(H_2\) Optimality}) Consider the system:
	\begin{equation}\label{eq:2.4}
		\begin{cases}
			dX(s) &= \Big[ A_1(s) + C_1(s)V(s) \Big]X(s) + \Big[ \bar{A}_1(s) + C_1(s)\bar{V}(s) + \bar{C}_1(s)\widetilde{V}(s) \Big] \mathbb{E}[X(s)] \\
			&\quad + B_1(s)u(s) + \bar{B}_1(s)\mathbb{E}[u(s)] + b(s) +C_1(s)V_0(s)+ \bar{C}_1(s)\mathbb{E}[V_0(s)] \, ds \\
			&\quad + \Big[ A_2(s) + C_2(s)V(s) \Big]X(s) + \Big[ \bar{A}_2(s) + C_2(s)\bar{V}(s) + \bar{C}_2(s)\widetilde{V}(s) \Big] \mathbb{E}[X(s)] \\
			&\quad + B_2(s)u(s) + \bar{B}_2(s)\mathbb{E}[u(s)] + \sigma(s) +C_2(s)V_0(s)+ \bar{C}_2(s)\mathbb{E}[V_0(s)] \, dW(s), \\
			z(s) &= \begin{pmatrix} Q(s)X(s) \\ N_1(s)u(s) \end{pmatrix}, \quad N_1(s)^\top N_1(s) = I, \\
			X(t) &= \xi.
		\end{cases}
	\end{equation}
	where \(\widetilde{V}(s)=V(s)+\bar{V}(s)\), and define the cost functional
	\[
	J_2(t, \xi; u(\cdot), v(\cdot)) \triangleq \|z\|_{[t, T]}^2 = \mathbb{E} \int_t^T |z(s)|^2 \, ds.
	\]
	
	Consider the closed-loop control laws in linear feedback form:
	\[
	\begin{aligned}
		u^*(s) &= U(s)X^*(s) + \bar{U}(s)\mathbb{E}[X^*(s)] + U_0(s), \\
		v^*(s) &= V(s)X^*(s) + \bar{V}(s)\mathbb{E}[X^*(s)],
	\end{aligned}
	\]
	where \(X^*(\cdot)\) is the state trajectory corresponding to the control pair \((u^*(\cdot), v^*(\cdot))\). These controls minimize the cost functional \(J_2\) in the sense that for any \(u(\cdot) \in L_{\mathbb{F}}^2(t, T; \mathbb{R}^{n_u})\),
	\[
J_2(t, \xi; u^*(\cdot), v^*(\cdot)) \leq J_2\left(t, \xi; u(\cdot), V(\cdot)X(\cdot) + \bar{V}(\cdot)\mathbb{E}[X(\cdot)] + V_0(\cdot) \right).
\]
\end{itemize}
\end{defn}

	To ensure condition (i) in Definition~\ref{def:1}, namely that the linear perturbation operator \(\mathbf{L}\) satisfies \(\|\mathbf{L}\| < \gamma\) for some \(\gamma > 0\), we introduce the performance functional
\[
J_1(t, \xi; u(\cdot), v(\cdot)) = \mathbb{E} \int_t^T \left[ \gamma^2  |v(s)|^2 - \langle Q(s)^\top Q(s) X(s), X(s) \rangle -  |u(s)|^2 \right] ds,
\]
subject to the state equation~\eqref{eq:2.1}. In particular, when the affine terms \( b(\cdot), \sigma(\cdot) = 0 \), the performance index is denoted by \( J_1^0(t,\xi;u(\cdot),v(\cdot)) \) and \( J_2^0(t,\xi;u(\cdot),v(\cdot)) \) .

Minimizing \(J_1\) enforces the \(H_\infty\) robustness criterion by bounding the worst-case energy gain from disturbance \(v(\cdot)\) to output \(z(\cdot)\). Meanwhile, the previously defined \(J_2\) characterizes the \(H_2\) performance, representing nominal output energy. Minimizing \(J_2\) achieves optimal regulation and energy efficiency. Combining these leads to a unified \(H_2/H_\infty\) control problem, which can be formulated as
the following {nonzero-sum differential game}.

\begin{pro}\textbf{[Problem (NSDG-\(H_2/H_\infty\))]}\label{def:NSDG}
We seek a pair of closed-loop strategies
\[
\left(U(\cdot), \bar{U}(\cdot), U_0(\cdot); V(\cdot), \bar{V}(\cdot), V_0(\cdot)\right) \in \mathcal{N}^2[t, T] \times \mathcal{M}^2[t, T],
\]
such that for any initial state \(\xi \in L^2(\Omega,\mathbb{P}, \mathcal{F}_t; \mathbb{R}^n)\), the following two inequalities are satisfied:\\
(i)	For all \(v(\cdot) \in L^2_{\mathbb{F}}(t, T; \mathbb{R}^{n_v})\),\begin{equation}\label{eq:2.5}
	J_1(t, \xi; u^*(\cdot), v^*(\cdot)) \leq J_1\left(t, \xi; U(\cdot)X(\cdot) + \bar{U}(\cdot)\mathbb{E}[X(\cdot)] + U_0(\cdot), v(\cdot)\right).
	\end{equation}
(ii)	For all \(u(\cdot) \in L^2_{\mathbb{F}}(t, T; \mathbb{R}^{n_u})\), \begin{equation} \label{eq:2.6}
	J_2(t, \xi; u^*(\cdot), v^*(\cdot)) \leq J_2\left(t, \xi; u(\cdot), V(\cdot)X(\cdot) + \bar{V}(\cdot)\mathbb{E}[X(\cdot)] + V_0(\cdot)\right).
	\end{equation}

Here, the closed-loop controls $(u^*(\cdot), v^*(\cdot))$ are defined as
\begin{equation}\label{eq:closed_loop_controls}
\begin{aligned}
u^*(\cdot) &= U(\cdot)X^*(\cdot) + \bar{U}(\cdot)\mathbb{E}[X^*(\cdot)] + U_0(\cdot), \\
v^*(\cdot) &= V(\cdot)X^*(\cdot) + \bar{V}(\cdot)\mathbb{E}[X^*(\cdot)] + V_0(\cdot).
\end{aligned}
\end{equation}

\end{pro}

 In \eqref{eq:closed_loop_controls},  \(X^*(\cdot)\) is the corresponding state trajectory under $(u^*(\cdot), v^*(\cdot))$. In \eqref{eq:2.5}, \(X(\cdot)\) denotes the state trajectory of system~\eqref{eq:2.1} under the control
\(u(\cdot) = U(\cdot)X(\cdot) + \bar{U}(\cdot)\mathbb{E}[X(\cdot)] + U_0(\cdot),\)
and an arbitrary disturbance input \(v(\cdot)\in L^2_{\mathbb{F}}(t, T; \mathbb{R}^{n_v})\).

Similarly, in \eqref{eq:2.6}, \(X(\cdot)\) is the trajectory of system~\eqref{eq:2.1} corresponding to
\(v(\cdot) = V(\cdot)X(\cdot) + \bar{V}(\cdot)\mathbb{E}[X(\cdot)] + V_0(\cdot)\),
and an arbitrary control input \(u(\cdot)\in L^2_{\mathbb{F}}(t, T; \mathbb{R}^{n_u})\).

A closed-loop strategy tuple
\[
\left(U(\cdot), \bar{U}(\cdot), U_0(\cdot); V(\cdot), \bar{V}(\cdot), V_0(\cdot)\right)
\]
that satisfies the above conditions is called a \emph{closed-loop Nash equilibrium strategy}, and the corresponding pair \((u^*(\cdot), v^*(\cdot))\) is referred to as a \emph{closed-loop Nash equilibrium}.

\section{A Bounded Real Lemma via Coupled Riccati Equations for Mean-Field Stochastic Systems}
In this section, we establish a stochastic bounded real lemma for mean-field systems with affine terms. This result forms a cornerstone in the analysis of stochastic system stability and robust control. In particular, it provides necessary and sufficient conditions under which the \(H_\infty\) norm of the system is bounded by a prescribed disturbance attenuation level \(\gamma\), thereby guaranteeing robustness against stochastic perturbations. The characterization is based on the solvability of a system of coupled Riccati differential equations, and serves as a foundation for the synthesis of \(H_2/H_\infty\) controllers as well as for ensuring the stability of the closed-loop system.

We consider the mean-field stochastic system described in equation~\eqref{eq:2.1} under zero control input, i.e., \(u(\cdot) = 0\) over the interval \([t, T]\). For an initial condition \((t, \xi) \in [0, T] \times L^2(\Omega, \mathbb{P}, \mathcal{F}_t; \mathbb{R}^n)\), the system reduces to the following controlled stochastic differential equation:
\begin{equation}\label{eq:1}
\left\{
\begin{aligned}
dX(s) &= \Big[A_1(s)X(s) + \bar{A}_1(s)\mathbb{E}[X(s)] + C_1(s)v(s) + \bar{C}_1(s)\mathbb{E}[v(s)] + b(s)\Big] ds \\
&\quad + \Big[A_2(s)X(s) + \bar{A}_2(s)\mathbb{E}[X(s)] + C_2(s)v(s) + \bar{C}_2(s)\mathbb{E}[v(s)] + \sigma(s)\Big] dW(s), \quad s \in [t, T], \\
z(s) &= Q(s)X(s), \\
X(t) &= \xi.
\end{aligned}
\right.
\end{equation}

This equation describes the evolution of the state \(X(s)\) under the influence of stochastic disturbances \(v(s)\), system dynamics, noise, and affine terms \(b(s)\), \(\sigma(s)\). The output \(z(s)\) is defined as a linear transformation of the state through the matrix \(Q(s)\).

To quantify the system's sensitivity to external disturbances, we define the following perturbation operator and associated norm.

\begin{defn}\label{def:2}
For the system~\eqref{eq:1}, define the operator \(L\) : \(L^2_{\mathbb{F}}(t, T; \mathbb{R}^{n_v})\) \(\rightarrow L^2_{\mathbb{F}}(t, T; \mathbb{R}^{n_z})\), by
\[
(Lv)(s) \triangleq z(s) = Q(s)X(s) \big|_{X(t) = 0,\ b = 0,\ \sigma = 0}, \quad s \in [t, T].
\]
The induced operator norm of \(L\), known as the system's \(H_\infty\) norm, is given by:
\begin{equation} \label{eq:3.200}
\|L\| \triangleq \sup_{\substack{v(\cdot) \in L^2_{\mathbb{F}}(t, T; \mathbb{R}^{n_v}) \\ v \neq 0}} \frac{\|Lv\|}{\|v\|},
\end{equation}
where
\[
\|Lv\| \triangleq \left( \mathbb{E} \int_t^T X(s)^\top Q(s)^\top Q(s) X(s) \, ds \right)^{1/2}, \quad
\|v\| \triangleq \left( \mathbb{E} \int_t^T v(s)^\top v(s) \, ds \right)^{1/2}.
\]
The \(H_\infty\) norm measures the worst-case amplification from the disturbance \(v(\cdot)\) to the output \(z(\cdot)\). A smaller norm implies enhanced robustness, indicating that the system's output remains relatively insensitive to perturbations.
\end{defn}

In stochastic \(H_\infty\) control, robustness analysis is often reformulated as the minimization of a quadratic cost functional. For an initial condition \((t, \xi) \in [0, T] \times L^2(\Omega, \mathbb{P}, \mathcal{F}_t; \mathbb{R}^n)\), the robust performance index is defined as:
\begin{equation}\label{eq:Hinf-cost}
\begin{aligned}
\min_{v(\cdot) \in L^2_{\mathbb{F}}(t, T; \mathbb{R}^{n_v})} J_1(t, \xi; 0, v(\cdot)) &\triangleq \min_{v(\cdot)\in L^2_{\mathbb{F}}(t, T; \mathbb{R}^{n_v})} \mathbb{E} \int_t^T \left[ \gamma^2 |v(s)|^2 - |z(s)|^2 \right] ds \\
&= \min_{v(\cdot)\in L^2_{\mathbb{F}}(t, T; \mathbb{R}^{n_v})} \mathbb{E} \int_t^T \left[ \langle \gamma^2 v(s), v(s) \rangle - \langle Q(s)^\top Q(s) X(s), X(s) \rangle \right] ds.
\end{aligned}
\end{equation}

In this formulation, \(\gamma^2 |v(s)|^2\) represents the weighted energy of the disturbance, while \(|z(s)|^2\) quantifies the energy of the system output. By minimizing \(J_1\), one seeks an optimal disturbance input \(v(\cdot)\) that achieves the best trade-off between disturbance energy and output energy. This approach can provide an equivalent characterization of the stochastic \(H_\infty\) norm bound and leads to what is referred to as the \textbf{\(H_\infty\) Optimal Disturbance Suppression Problem}.

The cost functional~\eqref{eq:Hinf-cost} corresponds to an indefinite linear-quadratic form, due to the positive weighting on the disturbance and the negative weighting on the state. This leads to an indefinite LQ formulation.

 To characterize the optimality conditions and facilitate robustness analysis, we derive a system of CDREs, which arise from the underlying stochastic LQ structure.

The CDREs are given by:
\begin{equation}\label{eq:3.2}
\begin{cases}
\dot{P} + PA_1 + A_1^\top P + A_2^\top P A_2 - Q^\top Q- (PC_1 + A_2^\top P C_2)(\gamma^2 I + C_2^\top P C_2)^{-1}(C_1^\top P + C_2^\top P A_2) = 0, \\[1ex]
\dot{\Pi} + \Pi\widetilde{A}_1 + \widetilde{A}_1^\top \Pi + \widetilde{A}_2^\top P \widetilde{A}_2 - Q^\top Q - (\Pi\widetilde{C}_1 + \widetilde{A}_2^\top P \widetilde{C}_2)(\gamma^2 I + \widetilde{C}_2^\top P \widetilde{C}_2)^{-1}(\widetilde{C}_1^\top \Pi + \widetilde{C}_2^\top P \widetilde{A}_2) = 0, \\[1ex]
\Lambda^\gamma(P) := \gamma^2 I + C_2^\top P C_2 \geq \delta I, \quad
\bar{\Lambda}^\gamma(P) := \gamma^2 I + \widetilde{C}_2^\top P \widetilde{C}_2 \geq \delta I, \quad \text{for some } \delta > 0, \\[1ex]
P(T) = 0, \quad \Pi(T) = 0,
\end{cases}
\end{equation}
where \(\widetilde{A}_i := A_i + \bar{A}_i\) and \(\widetilde{C}_i := C_i + \bar{C}_i\), for \(i = 1, 2\). Here, we have omitted the explicit time dependence of \(P(s)\), \(\Pi(s)\), and other related functions throughout the section unless required for clarity.

The CDREs describe the backward-in-time evolution of the matrix-valued functions \(P(s)\) and \(\Pi(s)\), which encode information about the variance and the mean of the system state, respectively. The uniform positive definiteness conditions \(\Lambda^\gamma(P) \geq \delta I\) and \(\bar{\Lambda}^\gamma(P) \geq \delta I\) ensure the invertibility of the associated gain matrices and the well-posedness of the feedback control law. The terminal conditions \(P(T) = 0\) and \(\Pi(T) = 0\) stem from the fact that the performance functional~\eqref{eq:Hinf-cost} does not include any terminal penalty on the state.

For notational convenience, we introduce the following brief expressions:
\begin{equation}\label{eq:3.3}
\begin{cases}
\Phi(P) := PC_1 + A_2^\top PC_2, \quad \bar{\Phi}(P, \Pi) := \Pi \widetilde{C}_1 + \widetilde{A}_2^\top P \widetilde{C}_2, \\[1ex]
\Lambda^\gamma(P) := \gamma^2 I + C_2^\top P C_2, \quad \bar{\Lambda}^\gamma(P) := \gamma^2 I + \widetilde{C}_2^\top P \widetilde{C}_2, \\[1ex]
\varphi := C_1^\top \eta + C_2^\top (P\sigma + \zeta), \quad \bar{\varphi} := \widetilde{C}_1^\top \mathbb{E}[X] + \widetilde{C}_2^\top (P\mathbb{E}[\sigma] + \mathbb{E}[\zeta]).
\end{cases}
\end{equation}

Now we can present the stochastic bounded real lemma for mean-field stochastic systems with affine terms. The following result establishes a necessary and sufficient condition for the \(H_\infty\) norm of the system to be less than a prescribed disturbance attenuation level \(\gamma\), based on the solvability of the CDREs \eqref{eq:3.2}.

\begin{lem}\label{lem:3.3}    Let Assumptions \ref{ass:3.1}-\ref{ass:3.3} holds.
Given a disturbance attenuation level \(\gamma > 0\), the system~\eqref{eq:1} satisfies \(\|L\| < \gamma\) if and only if the coupled Riccati equations~\eqref{eq:3.2} admit a pair of solutions \((P(\cdot), \Pi(\cdot)) \in C([t, T]; \mathbb{S}^n) ^2\).

Moreover, the following backward stochastic differential equation (BSDE):
\begin{equation} \label{eq:3.40}
\begin{cases}
d\eta(s) = - \left[ (A_1 - C_1 \Lambda^\gamma(P)^{-1} \Phi(P))^\top \eta(s) + (A_2 - C_2 \Lambda^\gamma(P)^{-1} \Phi(P))^\top (P\sigma + \zeta) + P b \right] ds + \zeta(s) dW(s), \\
\eta(T) = 0,
\end{cases}
\end{equation}
and the ordinary differential equation (ODE):
\begin{equation} \label{eq:3.50}
\begin{cases}
\dot{\bar{\eta}}(s) + \left[ \widetilde{A}_1 - \widetilde{C}_1 \bar{\Lambda}^\gamma(P)^{-1} \bar{\Phi}(P, \Pi) \right]^\top \bar{\eta}(s) \\
\qquad - \bar{\Phi}(P, \Pi) \bar{\Lambda}^\gamma(P)^{-1} \widetilde{C}_2 (P\mathbb{E}[\sigma] + \mathbb{E}[\zeta]) + \widetilde{A}_2^\top (P\mathbb{E}[\sigma] + \mathbb{E}[\zeta]) + \Pi \mathbb{E}[b] = 0, \\
\bar{\eta}(T) = 0,
\end{cases}
\end{equation}
admit solutions \((\eta(\cdot), \zeta(\cdot)) \in L^2_{\mathbb{F}}(\Omega; C([t, T]; \mathbb{R}^n)) \times L^2_{\mathbb{F}}(t, T; \mathbb{R}^n)\) and \(\bar{\eta}(\cdot) \in L^2(t, T; \mathbb{R}^{n})\).

The corresponding worst-case disturbance input for the \textbf{\(H_\infty\) optimal disturbance suppression problem} is given by:
\begin{equation}\label{eq:3.6}
\begin{aligned}
v^*(s) &= - \Lambda^\gamma(P)^{-1} \Phi(P)^\top (X(s) - \mathbb{E}[X(s)]) \\
&\quad - \bar{\Lambda}^\gamma(P)^{-1} \bar{\Phi}(P, \Pi)^\top \mathbb{E}[X(s)] - \Lambda^\gamma(P)^{-1} \varphi(s) - \bar{\Lambda}^\gamma(P)^{-1} \bar{\varphi}(s).
\end{aligned}
\end{equation}

Furthermore, the corresponding optimal cost is expressed as:
\begin{equation}\label{eq:3.7}
\begin{aligned}
J_1(t, \xi; 0, v^*(\cdot)) =\ & \mathbb{E} \left[ \langle P(\xi - \mathbb{E}[\xi]) + 2\eta(t), \xi - \mathbb{E}[\xi] \rangle \right] + \langle \Pi \mathbb{E}[\xi] + 2\bar{\eta}(t), \mathbb{E}[\xi] \rangle \\
& + \mathbb{E} \int_t^T \Big\{ \langle P\sigma, \sigma \rangle + 2\langle \eta, b - \mathbb{E}[b] \rangle + 2\langle \zeta, \sigma \rangle + 2\langle \bar{\eta}, \mathbb{E}[b] \rangle \\
& \qquad - \langle \varphi, \Lambda^\gamma(P)^{-1} \varphi \rangle - \langle \bar{\varphi}, \bar{\Lambda}^\gamma(P)^{-1} \bar{\varphi} \rangle \Big\} ds.
\end{aligned}
\end{equation}
\end{lem}

\begin{proof}
We recall from Definition~\ref{def:2} that the \(H_\infty\) norm of the system is characterized by the input-output operator \(L: v(\cdot) \mapsto z(\cdot)\) under zero control input. According to~\cite{wang2023stochastic}, the condition \(\|L\| < \gamma\) holds if and only if the CDREs \eqref{eq:3.2} admit a solution \((P(\cdot), \Pi(\cdot)) \in C([t, T]; \mathbb{S}^n)^2\) satisfying the uniform positivity conditions \(\Lambda^\gamma(P), \bar{\Lambda}^\gamma(P) \geq \delta I\), for some \(\delta > 0\).

Given such solutions, the BSDE~\eqref{eq:3.40} and the ODE~\eqref{eq:3.50} admit unique solutions \((\eta(\cdot), \zeta(\cdot)) \in L^2_{\mathbb{F}}(\Omega; C([t, T]; \mathbb{R}^n)) \times L^2_{\mathbb{F}}(t, T; \mathbb{R}^n)\) and \(\bar{\eta}(\cdot) \in C([t, T]; \mathbb{R}^n)\) with terminal conditions \(\eta(T) = 0\), \(\bar{\eta}(T) = 0\).

Define the centered variables
\(\bar{X}(s) := X(s) - \mathbb{E}[X(s)], \quad \bar{Z}(s) := v(s) - \mathbb{E}[v(s)].\)
Then, the dynamics of \(X(\cdot)\) decompose into two parts. The mean satisfies
\[
\begin{cases}
d\mathbb{E}[X(s)] = \left[\widetilde{A}_1(s)\mathbb{E}[X(s)] + \widetilde{C}_1(s)\mathbb{E}[v(s)] + \mathbb{E}[b(s)]\right] ds, \\
\mathbb{E}[X(t)] = \mathbb{E}[\xi],
\end{cases}
\]
while the fluctuation \(\bar{X}(s)\) evolves according to
\[
\begin{cases}
d\bar{X}(s) = \left[A_1(s)\bar{X}(s) + C_1(s)\bar{Z}(s) + b(s) - \mathbb{E}[b(s)]\right] ds \\
\qquad\qquad + \left[A_2(s)\bar{X}(s) + \widetilde{A}_2(s)\mathbb{E}[X(s)] + C_2(s)\bar{Z}(s) + \widetilde{C}_2(s)\mathbb{E}[v(s)] + \sigma(s)\right] dW(s), \\
\bar{X}(t) = \xi - \mathbb{E}[\xi].
\end{cases}
\]

We rewrite the cost functional \(J_1(t, \xi; 0, v(\cdot))\) as
	\[
\begin{split}
	J_1(t, \xi; 0, v(\cdot)) &= \mathbb{E}\int_t^T \left\{\langle\gamma^2\bar{Z}(s), \bar{Z}(s)\rangle - \langle Q(s)^\top Q(s)\bar{X}(s), \bar{X}(s)\rangle\right\}ds \\
	&\quad + \mathbb{E}\int_t^T \left\{\langle\gamma^2\mathbb{E}[v(s)], \mathbb{E}[v(s)]\rangle - \langle Q(s)^\top Q(s)\mathbb{E}[X(s)], \mathbb{E}[X(s)]\rangle\right\}ds.
\end{split}
\]

By applying It\^{o}'s formula to the processes \(\langle P(s)\bar{X}(s) + 2\eta(s), \bar{X}(s)\rangle\) and \(\langle \Pi(s)\mathbb{E}[X(s)] + 2\bar{\eta}(s), \mathbb{E}[X(s)] \rangle\), and invoking the CDREs~\eqref{eq:3.2}, one obtains an identity that leads to the completion-of-squares argument. Specifically, we derive
\[
\begin{aligned}
	&J_1(t, \xi; 0, v(\cdot)) -  \mathbb{E}\big\langle P(t)(\xi - \mathbb{E}[\xi]) + 2\eta(t), \xi - \mathbb{E}[\xi]\big\rangle - \big\langle \Pi\mathbb{E}[\xi] + 2\bar{\eta}(t), \mathbb{E}[\xi]\big\rangle \\
	 =& \mathbb{E} \int_t^T \Big\| \bar{Z}(s) - \Lambda^\gamma(P)^{-1} \Phi(P)^\top \bar{X}(s) - \Lambda^\gamma(P)^{-1} \varphi(s) \Big\|^2_{\Lambda^\gamma(P)} ds \\
	&+ \int_t^T \Big\| \mathbb{E}[v(s)] - \bar{\Lambda}^\gamma(P)^{-1} \bar{\Phi}(P,\Pi)^\top \mathbb{E}[X(s)] - \bar{\Lambda}^\gamma(P)^{-1} \bar{\varphi}(s) \Big\|^2_{\bar{\Lambda}^\gamma(P)} ds\\&+\mathbb{E}\int_t^T \bigg\{ \langle P\sigma, \sigma\rangle + 2\langle\eta, b - \mathbb{E}[b]\rangle + \langle\zeta, b\rangle + 2\langle\bar{\eta}, \mathbb{E}[b]\rangle  - \langle\Lambda^\gamma(P)^{-1}\varphi, \varphi\rangle - \langle\bar{\Lambda}^\gamma(P)^{-1}\bar{\varphi}, \bar{\varphi}\rangle \bigg\}ds.
\end{aligned}
\]

Thus, the cost is minimized by the choice of \(v^*(\cdot)\) in~\eqref{eq:3.6}, and  and  it easy to check that the minimal value equals the expression in~\eqref{eq:3.7}.
\end{proof}
	
	\section{Mean-Field Stochastic $H_2/H_\infty$ Control problem with affine terms}
In this section, we are dedicated to establishing a rigorous mathematical framework for the mixed $H_2/H_\infty$ control of mean-field stochastic systems under affine perturbations. The core of our study is centered on the characterization of the decoupling equations for both open-loop and closed-loop strategies. The definition of the closed-loop strategy has been clearly provided  in    Definition~\ref{def:1} and Problem~\ref{def:NSDG}. In the following, we will first give a strict definition of the open-loop strategy and conduct an in-depth exploration of the specific characterization of its decoupling equation.

\subsection{Open-Loop Solvability}
\label{subsec:openloop}
	
\begin{defn}\label{def:4.1}
Let $\xi \in L^2(\Omega, \mathbb{P}, \mathcal{F}_t; \mathbb{R}^n)$. A pair \((u^*(\cdot), v^*(\cdot))
 \in L^2_{\mathbb{F}}(t, T; \mathbb{R}^{n_u}) \times L^2_{\mathbb{F}}(t, T; \mathbb{R}^{n_v})\) is said to
 constitute an \emph{open-loop Nash equilibrium strategy}  with respect to the initial state \(\xi\)
  if the following conditions are satisfied for the system \eqref{eq:2.1}:\\
(i) For all \(v(\cdot) \in L^2_{\mathbb{F}}(t, T; \mathbb{R}^{n_v})\), the inequality
    \(
    J_1(t, \xi; u^*(\cdot), v^*(\cdot)) \leq J_1(t, \xi; u^*(\cdot), v(\cdot))
    \)
    holds.\\
(ii) For all \(u(\cdot) \in L^2_{\mathbb{F}}(t, T; \mathbb{R}^{n_u})\), the inequality
    \(
    J_2(t, \xi; u^*(\cdot), v^*(\cdot)) \leq J_2(t, \xi; u(\cdot), v^*(\cdot))
    \)
    holds.\\
In this case, the control \(u^*(\cdot)\) is called an open-loop \(H_2/H_\infty\) controller, and \(v^*(\cdot)\) represents the corresponding open-loop worst-case disturbance.
\end{defn}

\begin{rmk}
The notion of open-loop \(H_2/H_\infty\) control was originally introduced in \cite{chen2004stochastic} in the context of stochastic It\^{o } systems. From a game-theoretic perspective, the pair \((u^*(\cdot), v^*(\cdot))\) represents a Nash equilibrium strategy of a two-player non-zero-sum differential game.

It is important to highlight that the open-loop \(H_2/H_\infty\) control strategy is inherently tied to the specific initial condition \(\xi\). In contrast, the closed-loop strategy defined in Definition~\ref{def:1} applies uniformly for all admissible initial states. The concept of open-loop \(H_2/H_\infty\) control studied here extends significantly beyond earlier formulations, including those in \cite{chen2004stochastic} and \cite{wang2023stochastic}, by incorporating mean-field interactions and affine terms within the system dynamics.
\end{rmk}

By introducing a mean-field forward-backward stochastic differential equation (MF-FBSDE, for short), we further investigate the open-loop \(H_2/H_\infty\) control problem. Specifically, we consider the MF-FBSDE associated with the state process \(X(\cdot) = X(\cdot; t, \xi, u(\cdot), v(\cdot))\) governed by the system \eqref{eq:2.1}:
\begin{equation}\label{eq:4.1}
	\begin{cases}
		\begin{aligned}
			dX(s) &= \big\{A_1(s)X(s) + \bar{A}_1(s)\mathbb{E}[X(s)] + B_1(s)u(s) + \bar{B}_1(s)\mathbb{E}[u(s)] \\
			&\quad + C_1(s)v(s) + \bar{C}_1(s)\mathbb{E}[v(s)] + b(s)\big\}ds \\
			&\quad + \big\{A_2(s)X(s) + \bar{A}_2(s)\mathbb{E}[X(s)] + B_2(s)u(s) + \bar{B}_2(s)\mathbb{E}[u(s)] \\
			&\quad + C_2(s)v(s) + \bar{C}_2(s)\mathbb{E}[v(s)] + \sigma(s)\big\}dW(s), \\[1ex]
			dY_1(s) &= -\big\{A_1(s)^\top Y_1(s) + \bar{A}_1(s)^\top\mathbb{E}[Y_1(s)] + A_2(s)^\top Z_1(s) \\
			&\quad + \bar{A}_2(s)^\top\mathbb{E}[Z_1(s)] - Q(s)^\top Q(s)X(s)\big\}ds + Z_1(s)dW(s), \\[1ex]
			dY_2(s) &= -\big\{A_2(s)^\top Y_2(s) + \bar{A}_1(s)^\top\mathbb{E}[Y_2(s)] + A_2(s)^\top Z_2(s) \\
			&\quad + \bar{A}_2(s)^\top\mathbb{E}[Z_2(s)] + Q(s)^\top Q(s)X(s)\big\}ds + Z_2(s)dW(s), \\[1ex]
			X(t) &= \xi, \quad Y_1(T) = 0, \quad Y_2(T) = 0, \quad s \in [t, T].
		\end{aligned}
	\end{cases}
\end{equation}

The following theorem establishes necessary and sufficient conditions for the existence of an open-loop \(H_2/H_\infty\) control strategy. This result provides a key characterization of the solvability of the open-loop problem in terms of the disturbance attenuation requirement and the solvability of the associated MF-FBSDE \eqref{eq:4.1} with affine terms.
	\begin{thm}\label{thm:4.3}   Let Assumptions \ref{ass:3.1}-\ref{ass:3.3} holds.
Consider the system \eqref{eq:2.1} with an initial condition \(\xi \in L^2(\Omega, \mathbb{P}, \mathcal{F}_t; \mathbb{R}^n)\) and a prescribed disturbance attenuation level \(\gamma > 0\). Then a pair \((u^*(\cdot), v^*(\cdot)) \in L^2_{\mathbb{F}}(t, T; \mathbb{R}^{n_u}) \times L^2_{\mathbb{F}}(t, T; \mathbb{R}^{n_v})\) constitutes an \emph{open-loop Nash equilibrium strategy}  with respect to \(\xi\) if and only if the following conditions hold:\\
(i) \textbf{Disturbance Attenuation:} The induced operator norm of \(L\), as defined in Definition~\ref{def:2}, satisfies
    \[
    \|L\| \leq \gamma.
    \]
(ii) \textbf{MF-FBSDE Solvability and Stationarity:} The MF-FBSDE \eqref{eq:4.1} admits a solution
    \[
    (X^*(\cdot), Y_i^*(\cdot), Z_i^*(\cdot)) \in L^2_{\mathbb{F}}(\Omega; C(t, T; \mathbb{R}^n))^2 \times L^2_{\mathbb{F}}(t, T; \mathbb{R}^{n}),
    \]
    together with control processes \((u^*(\cdot), v^*(\cdot))\), such that the following stationarity conditions hold for almost every \(s \in [t, T]\):
    \begin{equation}\label{eq:4.2}
    \begin{cases}
        C_1(s)^\top Y^*_1(s) + C_2(s)^\top Z^*_1(s) + \bar{C}_1(s)^\top \mathbb{E}[Y^*_1(s)] + \bar{C}_2(s)^\top \mathbb{E}[Z^*_1(s)] + \gamma^2 v^*(s) = 0, \\[1ex]
        B_1(s)^\top Y^*_2(s) + B_2(s)^\top Z^*_2(s) + \bar{B}_1(s)^\top \mathbb{E}[Y^*_2(s)] + \bar{B}_2(s)^\top \mathbb{E}[Z^*_2(s)] + u^*(s) = 0.
    \end{cases}
    \end{equation}
\end{thm}

	\begin{proof}

\textbf{Sufficiency.} Assume that conditions (i) and (ii) in Theorem~\ref{thm:4.3} hold. We first prove that \(v^*(\cdot)\) is an open-loop worst-case disturbance.

For any \(v_1(\cdot) \in L^2_{\mathbb{F}}(t, T; \mathbb{R}^{n_v})\), let \(X_1(\cdot)\) denote the corresponding state of system \eqref{eq:2.1} with control \(u^*(\cdot)\) and disturbance \(v_1(\cdot)\). Define the variations:
\[
\Delta v(\cdot) = v_1(\cdot) - v^*(\cdot), \quad \Delta X_1(\cdot) = X_1(\cdot) - X^*(\cdot).
\]
Then \(\Delta X_1(\cdot)\) satisfies the linear MF-SDE:
\begin{equation*}
\begin{cases}
d(\Delta X_1(s)) = \{A_1(s)\Delta X_1(s) + \bar{A}_1(s)\mathbb{E}[\Delta X_1(s)] + C_1(s)\Delta v(s) + \bar{C}_1(s)\mathbb{E}[\Delta v(s)]\}ds \\[0.5em]
\qquad\qquad\qquad + \{A_2(s)\Delta X_1(s) + \bar{A}_2(s)\mathbb{E}[\Delta X_1(s)] + C_2(s)\Delta v(s) + \bar{C}_2(s)\mathbb{E}[\Delta v(s)]\}dW(s), \\[0.5em]
\Delta X_1(t) = 0,\qquad s\in[t,T].
\end{cases}
\end{equation*}

Applying It\^{o}'s formula to \(\langle \Delta X_1(s), Y_1^*(s) \rangle\) and using the adjoint equation, we obtain
\begin{equation}\label{eq:4.3}
\mathbb{E} \int_t^T \langle Q(s)^\top Q(s)\Delta X_1(s), X^*(s) \rangle ds = -\mathbb{E} \int_t^T \langle \Delta v(s), \Psi_1(s) \rangle ds,
\end{equation}
where \(\Psi_1(s) := C_1(s)^\top Y_1^*(s) + C_2(s)^\top Z_1^*(s) + \bar{C}_1(s)^\top \mathbb{E}[Y_1^*(s)] + \bar{C}_2(s)^\top \mathbb{E}[Z_1^*(s)]\).

Since \(\|L\| \leq \gamma\), the energy inequality yields
\[
J_1^0(t, 0; 0, \Delta v(\cdot)) = \mathbb{E} \int_t^T \left( \gamma^2 |\Delta v(s)|^2 - |Q(s)\Delta X_1(s)|^2 \right) ds \geq 0.
\]

Using this, we compute
\begin{equation}\label{eq:4.4}
	\begin{split}
&J_1(t, \xi; u^*(\cdot), v_1(\cdot)) - J_1(t, \xi; u^*(\cdot), v^*(\cdot))
\\=& J_1^0(t, 0; 0, \Delta v(\cdot)) + 2\mathbb{E} \int_t^T \left( \langle \gamma^2 \Delta v(s), v^*(s) \rangle - \langle Q^\top Q \Delta X_1(s), X^*(s) \rangle \right) ds.
	\end{split}
\end{equation}
Substituting \eqref{eq:4.3} into \eqref{eq:4.4} and using stationarity \eqref{eq:4.2}, we obtain
\[
J_1(t, \xi; u^*(\cdot), v_1(\cdot)) - J_1(t, \xi; u^*(\cdot), v^*(\cdot)) = J_1^0(t, 0; 0, \Delta v(\cdot)) \geq 0.
\]

We now show that \(u^*(\cdot)\) minimizes \(J_2(u(\cdot),v^*(\cdot))\). Define the variations \(\Delta u(\cdot) = u_1(\cdot) - u^*(\cdot)\), \(\Delta X_2(\cdot) = X_2(\cdot) - X^*(\cdot)\). Then \(\Delta X_2(\cdot)\) satisfies:
\begin{equation*}
\begin{cases}
d(\Delta X_2(s)) = \{A_1(s) \Delta X_2(s) + \bar{A}_1(s) \mathbb{E}[\Delta X_2(s)] + B_1(s) \Delta u(s) + \bar{B}_1(s) \mathbb{E}[\Delta u(s)]\} ds \\[0.5em]
\qquad\qquad\qquad + \{A_2(s) \Delta X_2(s) + \bar{A}_2(s) \mathbb{E}[\Delta X_2(s)] + B_2(s) \Delta u(s) + \bar{B}_2(s) \mathbb{E}[\Delta u(s)]\} dW(s), \\[0.5em]
\Delta X_2(t) = 0,\qquad s\in[t,T].
\end{cases}
\end{equation*}

Using It\^{o}'s formula and the adjoint process \(\langle \Delta X_2(s), Y_2^*(s) \rangle\), we get
\[
\mathbb{E} \int_t^T \langle Q^\top Q \Delta X_2(s), X^*(s) \rangle ds = \mathbb{E} \int_t^T \langle \Delta u(s), \Psi_2(s) \rangle ds,
\]
where \(\Psi_2 := B_1^\top Y_2^* + B_2^\top Z_2^* + \bar{B}_1^\top \mathbb{E}[Y_2^*] + \bar{B}_2^\top \mathbb{E}[Z_2^*]\).

Then,
\begin{align*}
J_2(t, \xi; u_1(\cdot), v^*(\cdot)) - J_2(t, \xi; u^*(\cdot), v^*(\cdot))
= \mathbb{E} \int_t^T  \langle\Delta u(s),\Delta u(s)\rangle + \langle Q \Delta X_2(s),Q \Delta X_2(s)\rangle ds \geq 0.
\end{align*}

This confirms that \(v^*(\cdot)\) is the open-loop worst-case disturbance and \(u^*(\cdot)\) is the corresponding optimal controller. Hence, the pair \((u^*(\cdot), v^*(\cdot))\) forms an open-loop \(H_2/H_\infty\) control strategy for system \eqref{eq:2.1} with initial state \(\xi\). \qed

\textbf{Necessity.} Suppose \((u^*(\cdot), v^*(\cdot))\) is an open-loop \(H_2/H_\infty\) equilibrium strategy. Then the associated MF-FBSDE \eqref{eq:4.1} admits a unique adapted solution  \[
(X^*(\cdot), Y_i^*(\cdot), Z_i^*(\cdot)) \in L^2_{\mathbb{F}}(\Omega; C(t, T; \mathbb{R}^n))^2 \times L^2_{\mathbb{F}}(t, T; \mathbb{R}^{n}).
\]

We apply the variational method. Consider a perturbation of the disturbance:
\[
\Delta v(\cdot) = \varepsilon v(\cdot), \quad \text{with } v(\cdot) \in L^2_{\mathbb{F}}(t, T; \mathbb{R}^{n_v}) \text{ arbitrary}, \, \varepsilon \in \mathbb{R}.
\]
The corresponding variation of the cost functional \(J_1\) is given by
\begin{align*}
&J_1(t, \xi; u^*(\cdot), v^*(\cdot) + \varepsilon v(\cdot)) - J_1(t, \xi; u^*(\cdot), v^*(\cdot))
\\=& \varepsilon^2 J_1^0(t, 0; 0, v(\cdot))
+ 2\varepsilon \mathbb{E} \int_t^T \left\langle v(s), \Psi_1(s) + \gamma^2 v^*(s) \right\rangle ds,
\end{align*}
where the process \(\Psi_1(\cdot)\) is defined as
\[
\Psi_1(s) := C_1(s)^\top Y_1^*(s) + C_2(s)^\top Z_1^*(s) + \bar{C}_1(s)^\top \mathbb{E}[Y_1^*(s)] + \bar{C}_2(s)^\top \mathbb{E}[Z_1^*(s)].
\]

Since \(v^*(\cdot)\) minimizes \(J_1\), the above variation must be non-negative for all \(\varepsilon \in \mathbb{R}\), implying the first-order stationarity condition:
\[
\Psi_1(s) + \gamma^2 v^*(s) = 0, \quad \text{a.e. } s \in [t, T], \ \text{a.s.}
\]

This is the first stationarity condition in \eqref{eq:4.2}. Substituting the stationarity condition into the expansion of \(J_1\), we obtain
\[
J_1(t, \xi; u^*(\cdot), v^*(\cdot) + \Delta v(\cdot)) - J_1(t, \xi; u^*(\cdot), v^*(\cdot)) = \varepsilon^2 J_1^0(t, 0; 0, \Delta v(\cdot)).
\]
Since \(v^*(\cdot)\) is optimal, the left-hand side is non-negative for all \(\varepsilon \in \mathbb{R}\) and \(v(\cdot) \neq 0\). This implies
\[
J_1^0(t, 0; 0, \Delta v(\cdot))\geq0.
\]	The inequality \(J_1^0(t, 0; 0, \Delta v(\cdot))\) ensures that the \(H_\infty\) norm of the system satisfies \(\|L\| \leq \gamma\).

Analogously, consider the perturbation in the control:
\[
\Delta u(\cdot) = \varepsilon u(\cdot), \quad u(\cdot) \in L^2_{\mathbb{F}}(t, T; \mathbb{R}^{n_u}) \text{ arbitrary},  \varepsilon \in \mathbb{R}.
\]
Then the variation of \(J_2\) satisfies:
\begin{align*}
&J_2(t, \xi; u^*(\cdot) + \varepsilon u(\cdot), v^*(\cdot)) - J_2(t, \xi; u^*(\cdot), v^*(\cdot))
\\=  & \varepsilon^2 \mathbb{E} \int_t^T \left( \|Q(s) u(s)\|^2 + \|u(s)\|^2 \right) ds + 2\varepsilon \mathbb{E} \int_t^T \left\langle u(s), \Psi_2(s) + u^*(s) \right\rangle ds,
\end{align*}
where
\[
\Psi_2(s) := B_1(s)^\top Y_2^*(s) + B_2(s)^\top Z_2^*(s) + \bar{B}_1(s)^\top \mathbb{E}[Y_2^*(s)] + \bar{B}_2(s)^\top \mathbb{E}[Z_2^*(s)].
\]

Optimality of \(u^*(\cdot)\) implies
\[
\Psi_2(s) + u^*(s) = 0, \quad \text{a.e. } s \in [t, T], \ \text{a.s.}
\]
Therefore, both optimality conditions lead to the stationarity conditions \eqref{eq:4.2}, which completes the necessity part of the proof. This completes the proof.
\end{proof}

	In this following, we focus on deriving a linear state-feedback representation for the  \emph{open-loop} \(H_2/H_\infty\) control strategy of  the \(H_2/H_\infty\) control problem.

	To simplify notation, let \(m = n_u + n_v\) and define the following augmented matrices (for \(i = 1, 2\)):
	\begin{equation}\label{eq:4.10}
		\left\{
		\begin{aligned}
			&B(\cdot) = \begin{pmatrix} B_1 & C_1 \end{pmatrix}(\cdot), \quad
			\bar{B}(\cdot) = \begin{pmatrix} \bar{B}_1 & \bar{C}_1 \end{pmatrix}(\cdot), \\
			&D(\cdot) = \begin{pmatrix} B_2 & C_2 \end{pmatrix}(\cdot), \quad
			\bar{D}(\cdot) = \begin{pmatrix} \bar{B}_2 & \bar{C}_2 \end{pmatrix}(\cdot), \\
			&\mathbf{A_i}(\cdot) = \begin{pmatrix} A_i & 0 \\ 0 & A_i \end{pmatrix}(\cdot), \quad
			\mathbf{\bar{A}_i}(\cdot) = \begin{pmatrix} \bar{A}_i & 0 \\ 0 & \bar{A}_i \end{pmatrix}(\cdot), \\
			&\mathbf{B}(\cdot) = \begin{pmatrix} B & 0 \\ 0 & B \end{pmatrix}(\cdot), \quad
			\mathbf{\bar{B}}(\cdot) = \begin{pmatrix} \bar{B} & 0 \\ 0 & \bar{B} \end{pmatrix}(\cdot), \\
			&\mathbf{D}(\cdot) = \begin{pmatrix} D & 0 \\ 0 & D \end{pmatrix}(\cdot), \quad
			\mathbf{\bar{D}}(\cdot) = \begin{pmatrix} \bar{D} & 0 \\ 0 & \bar{D} \end{pmatrix}(\cdot), \\
			&\mathbf{Q}(\cdot) = \begin{pmatrix} -Q^\top Q & 0 \\ 0 & Q^\top Q \end{pmatrix}(\cdot), \quad
			\mathbf{R} = \begin{pmatrix} I_{n_u} & 0 \\ 0 & \gamma^2 I_{n_v} \end{pmatrix}.
		\end{aligned}
		\right.
	\end{equation}
	
	The MF-FBSDE \eqref{eq:4.1} can then be compactly written as
	\begin{equation}\label{eq:4.13*}\left\{
		\begin{aligned}
			&dX^*(s) = \Big\{ A_1(s)X^*(s) + \bar{A}_1(s)\mathbb{E}[X^*(s)] + B(s)w^*(s) + \bar{B}(s)\mathbb{E}[w^*(s)] +b(s)\Big\}ds \\
			&\qquad\qquad + \Big\{ A_2(s)X^*(s) + \bar{A}_2(s)\mathbb{E}[X^*(s)] + D(s)w^*(s) + \bar{D}(s)\mathbb{E}[w^*(s)] +\sigma(s)\Big\}dW(s), \\
			&-d\mathbf{Y}^*(s) = \Big\{ \mathbf{A_1}(s)^\top\mathbf{Y}^*(s) + \mathbf{\bar{A}_1}(s)^\top\mathbb{E}[\mathbf{Y}^*(s)]\\
			&\qquad\qquad\qquad + \mathbf{A_2}(s)^\top\mathbf{Z}^*(s) + \mathbf{\bar{A}_2}^\top\mathbb{E}[\mathbf{Z}^*(s)] + \mathbf{Q}\mathbf{I}_n(s)X^*(s) \Big\}ds - \mathbf{Z}^*(s)dW(s), \\
			&X^*(t) = \xi,~\mathbf{Y}^*(T)=0,
		\end{aligned}\right.
	\end{equation}	
	with the stationarity conditions \eqref{eq:4.2} reformulated as:
	\begin{equation}\label{eq:4.12*}
		\mathbf{J}^\top \Big\{ \mathbf{B}(s)^\top\mathbf{Y}^*(s) + \mathbf{\widetilde{B}}(s)^\top\mathbb{E}[\mathbf{Y}^*(s)] + \mathbf{D}(s)^\top\mathbf{Z}^*(s) + \mathbf{\widetilde{D}}(s)^\top\mathbb{E}[\mathbf{Z}^*(s)] + \mathbf{R}\mathbf{I}_m(s)w^*(s) \Big\} = 0,
	\end{equation}
	where the augmented vectors and matrices are defined by:
	\begin{equation}
		\left\{
		\begin{aligned}
			&\mathbf{Y}^*(\cdot) = \begin{pmatrix} Y_2^*(\cdot) \\ Y_1^*(\cdot) \end{pmatrix}, \quad
			\mathbf{Z}^*(\cdot) = \begin{pmatrix} Z_2^*(\cdot) \\ Z_1^*(\cdot) \end{pmatrix}, \quad
			\mathbf{I}_n = \begin{pmatrix} I_{n\times n} \\ I_{n\times n} \end{pmatrix}, \\
			&\mathbf{I}_m = \begin{pmatrix} I_{m\times m} \\ I_{m\times m} \end{pmatrix}, \quad
			\mathbf{J} = \begin{bmatrix}
				I_{m_1\times m_1} & 0 \\
				0 & 0 \\
				0 & 0 \\
				0 & I_{m_2\times m_2}
			\end{bmatrix}, \widetilde{\mathbf{B}}(\cdot)\triangleq \mathbf{B}(\cdot) + \bar{\mathbf{B}}(\cdot),~\widetilde{\mathbf{D}}(\cdot) \triangleq \mathbf{D}(\cdot) + \bar{\mathbf{D}}(\cdot).
		\end{aligned}
		\right.
	\end{equation}	
	Currently, taking inspiration from \cite{sun2016open}, we are able to derive a closed-loop representation for the \emph{open-loop} \(H_2/H_\infty\) control strategy with respect to the initial state \(\xi\) if the following conditions are satisfied for the system \eqref{eq:2.1}. Specifically, an open- loop control strategy
\[	w^*(\cdot) = K^*(\cdot)\left(X^*(\cdot) - \mathbb{E}[X^*(\cdot)]\right) + \widetilde{K}^*(\cdot)\mathbb{E}[X^*(\cdot)] + \chi^*(\cdot),\] for certain \((K^*(\cdot),\widetilde{K}^*(\cdot),\chi^*(\cdot))\in  C(t, T ; \mathbb{R}^{m\times n})\times C(t, T ; \mathbb{R}^{m\times n})\times L_{\mathbb{F}}^{2}(t, T ; \mathbb{R}^{m})\), where \(w^*(\cdot)=(u^*(\cdot),v^*(\cdot)))^\top\). We present the result here.
		
	\begin{thm}
		Let Assumptions \ref{ass:3.1}-\ref{ass:3.3} holds. Let $\xi \in L^2(\Omega,\mathbb{P},\mathcal{F}_t;\mathbb{R}^n)$ and $\gamma > 0$. Suppose the following conditions are satisfied.
		
		(i)The system of coupled  differential Riccati equations (CDREs): \begin{equation}\label{eq:4.16*} \left\{\begin{aligned} &\dot{\mathbf{P}} + \mathbf{P}A_1 + \mathbf{A}_1^\top\mathbf{P} + \mathbf{A}_2^\top\mathbf{P}A_2 + \mathbf{QI}_n \\ &\quad - (\mathbf{P}B + \mathbf{A}_2^\top\mathbf{P}D)\mathbf{\Sigma}^{-1}\mathbf{J}^\top(\mathbf{B}^\top\mathbf{P} + \mathbf{D}^\top\mathbf{P}A_2) = 0, \\ &\dot{\mathbf{\Pi}} + \mathbf{\Pi}\widetilde{A}_1 + \widetilde{\mathbf{A}}_1^\top\mathbf{\Pi} + \widetilde{\mathbf{A}}_2^\top\mathbf{P}\widetilde{A}_2 + \mathbf{QI}_n \\ &\quad - (\mathbf{\Pi}\widetilde{B} + \widetilde{\mathbf{A}}_2^\top\mathbf{P}\widetilde{D})\bar{\mathbf{\Sigma}}^{-1}\mathbf{J}^\top(\widetilde{\mathbf{B}}^\top\mathbf{\Pi} + \widetilde{\mathbf{D}}^\top\mathbf{P}\widetilde{A}_2) = 0, \end{aligned}\right. \end{equation}
		admits a solution pair $$(\mathbf{P}(\cdot), \mathbf{\Pi}(\cdot)) \triangleq \begin{pmatrix} P_2 & \Pi_2 \\ P_1 & \Pi_1 \end{pmatrix}(\cdot)
		\in C([t, T]; \mathbb{S}^{2n \times 2n}) \times C([t, T]; \mathbb{S}^{2n \times 2n}),$$ where $$\widetilde{\mathbf{A}}_i(\cdot) \triangleq \mathbf{A}_i(\cdot) + \bar{\mathbf{A}}_i(\cdot
		),(i = 1,2),~\widetilde{\mathbf{B}}(\cdot)\triangleq \mathbf{B}(\cdot) + \bar{\mathbf{B}}(\cdot),~\widetilde{\mathbf{D}}(\cdot) \triangleq \mathbf{D}(\cdot) + \bar{\mathbf{D}}(\cdot),$$  and $$\mathbf{\Sigma} \triangleq \mathbf{J}^\top(\mathbf{R}\mathbf{I}_m + \mathbf{D}^\top\mathbf{P}D), \bar{\mathbf{\Sigma}} \triangleq \mathbf{J}^\top(\mathbf{R}\mathbf{I}_m + \widetilde{\mathbf{D}}^\top\mathbf{P}\widetilde{D}),$$ are invertible.
		
		(ii)   The backward stochastic differential equation (BSDE): \begin{equation}\label{eq:4.17*} \left\{\begin{split} d\eta(s) = -\Big[ \mathbf{A}_1(s)^\top\eta(s) &- (\mathbf{P}(s)B(s) + \mathbf{A}_2(s)^\top\mathbf{P}(s)D(s))\mathbf{\Sigma}(s)^{-1}\mathbf{J}^\top \\ &\cdot \{\mathbf{B}(s)^\top\eta(s) + \mathbf{D}(s)^\top[\zeta(s) + \mathbf{P}(s)\sigma(s)]\} \\ &+ \mathbf{A}_2(s)^\top[\zeta(s) + \mathbf{P}(s)\sigma(s)] + \mathbf{P}(s)b(s) \Big]ds + \zeta(s)dW(s), \\ \eta(T) = 0,~~~~~~~~~~~~~~~& \end{split}\right. \end{equation} has an adapted solution $(\eta(\cdot), \zeta(\cdot)) \triangleq \begin{pmatrix} \eta_2 ~~ \zeta_2 \\ \eta_1 ~~ \zeta_1 \end{pmatrix}\!  (\cdot)$ \(\in\) $L^2_{\mathbb{F}}(\Omega; C(t, T; \mathbb{R}^{2n})) \times L^2_{\mathbb{F}}(t, T; \mathbb{R}^{2n})$.
		
		(iii)  The ordinary differential equation (ODE): \begin{equation}\label{eq:4.18*} \left\{\begin{split} \dot{\bar{\eta}}(s) + \widetilde{\mathbf{A}}_1(s)^\top\bar{\eta}(s) &+ \widetilde{\mathbf{A}}_2(s)^\top\mathbb{E}[\zeta(s) + \mathbf{P}(s)\sigma(s)] + \mathbf{\Pi}(s)\mathbb{E}[b(s)] \\ &- (\mathbf{\Pi}(s)\widetilde{B}(s) + \widetilde{\mathbf{A}}_2(s)^\top\mathbf{P}(s)\widetilde{D}(s))\bar{\mathbf{\Sigma}}(s)^{-1} \\ &\cdot \{\widetilde{\mathbf{B}}(s)^\top\bar{\eta}(s) + \widetilde{\mathbf{D}}^\top\mathbb{E}[\zeta + \mathbf{P}(s)\sigma(s)]\} = 0, \\ \bar{\eta}(T) = 0,~~~~~~~~~~& \end{split}\right. \end{equation} admits a solution $\bar{\eta}(\cdot) \triangleq \begin{pmatrix} \bar{\eta}_2 \\ \bar{\eta}_1 \end{pmatrix}(\cdot)$ \(\in\) $C(t, T; \mathbb{R}^{2n})$.
		
		(iv)  The norm of the interference operator \(L\) satisfies \(\|L\| \leq \gamma\), where the norm \(\|L\|\) is consistent with Definition \ref{def:2}.
		
			Then the open-loop  \(H_2/H_\infty\) problem
		admits an \emph{open-loop Nash equilibrium strategy} \(w^*(\cdot)=(u^*(\cdot),v^*(\cdot)))^\top \in L_{\mathbb{F}}^2(t, T; \mathbb{R}^m) \) with the closed-loop representation:
		\[
		w^*(\cdot) = K^*(\cdot)\left(X^*(\cdot) - \mathbb{E}[X^*(\cdot)]\right) + \widetilde{K}^*(\cdot)\mathbb{E}[X^*(\cdot)] + \chi^*(\cdot),
		\]
		where
		\begin{equation}\label{eq:4.19}
			\begin{split}
				&K^*(\cdot) = -\mathbf{\Sigma}^{-1}\mathbf{J}^\top\left(\mathbf{B}^\top\mathbf{P} + \mathbf{D}^\top\mathbf{P}A_2\right) \in L_{\mathbb{F}}^{2}(t, T ; \mathbb{R}^{m\times n}),\\
				&\widetilde{K}^*(\cdot) = -\bar{\mathbf{\Sigma}}^{-1}\mathbf{J}^\top\left(\widetilde{\mathbf{B}}^\top\mathbf{\Pi} + \widetilde{\mathbf{D}}^\top\mathbf{P}\widetilde{A}_2\right) \in L_{\mathbb{F}}^{2}(t, T ; \mathbb{R}^{m\times n}),\end{split}
		\end{equation}
		and  \begin{equation}\label{eq:4.20*}
			\begin{split}
				\chi^*(\cdot) = -\mathbf{\Sigma}^{-1}\mathbf{J}^\top \Big[\mathbf{B}^\top(\eta - \mathbb{E}[\eta]) + \mathbf{D}^\top(\zeta - \mathbb{E}[\zeta]) + \mathbf{D}^\top\mathbf{P}(\sigma - \mathbb{E}[\sigma])\Big] - \bar{\mathbf{\Sigma}}^{-1}\mathbf{J}^\top \Big[\widetilde{\mathbf{B}}^\top\bar{\eta} + \widetilde{\mathbf{D}}^\top\mathbb{E}[\zeta] + \widetilde{\mathbf{D}}^\top\mathbf{P}\mathbb{E}[\sigma]\Big].
			\end{split}
		\end{equation}
	
	\end{thm}

	\begin{proof}
		If the Riccati equations \eqref{eq:4.16*} have a set of solutions \((\mathbf{P}(\cdot),\mathbf{\Pi}(\cdot))\) and the BSDE \eqref{eq:4.17*} and ODE \eqref{eq:4.18*} have solutions \((\eta(\cdot),\zeta(\cdot))\) and \(\bar{\eta}(\cdot)\), we can  define a pair  \(w^*(\cdot)=(u^*(\cdot),v^*(\cdot)))^\top\)
		as follows
		\begin{equation}
			\begin{split}
				w^*(\cdot) =&K^*(\cdot)\left(X^*(\cdot) - \mathbb{E}[X^*(\cdot)]\right) + \widetilde{K}^*(\cdot)\mathbb{E}[X^*(\cdot)] + \chi^*(\cdot)
				\\=&-\mathbf{\Sigma}^{-1}\mathbf{J}^\top(\mathbf{B}^\top\mathbf{P}+\mathbf{D}^\top\mathbf{P}A_2)(X^*-\mathbb{E}[X^*])-\mathbf{\bar{\Sigma}}^{-1}\mathbf{J}^\top(\mathbf{\widetilde{B}}^\top\mathbf{\Pi}+\mathbf{\widetilde{D}}^\top\mathbf{P}\widetilde{A}_2)\mathbb{E}[X^*]\\&
				-\mathbf{\Sigma}^{-1}\mathbf{J}^\top\{\mathbf{B}^{\top}(\eta-\mathbb{E}[\eta])+\mathbf{D}^\top(\zeta-\mathbb{E}[\zeta])+\mathbf{D}^\top\mathbf{P}(\sigma-\mathbb{E}[\sigma])\}\\&-\mathbf{\bar{\Sigma}}^{-1}\mathbf{J}^\top\{\mathbf{\widetilde{B}}^\top\bar{\eta}+\mathbf{\widetilde{D}}^\top\mathbb{E}[\zeta]+\mathbf{\widetilde{D}}^\top\mathbf{P}\mathbb{E}[\sigma]\}.
			\end{split}
		\end{equation}
		where \(X^*(\cdot)\) is the solution of the following MF-SDE:
		\begin{equation}\label{eq:MF-SDE}
	\left\{\begin{split}
				&dX^*(s) = \Big\{ A_1(s)X^*(s) + \bar{A}_1(s)\mathbb{E}[X^*(s)] + B(s)w^*(s) + \bar{B}(s)\mathbb{E}[w^*(s)]+b(s) \Big\}ds \nonumber\\
				&\qquad\quad + \Big\{ A_2(s)X^*(s) + \bar{A}_2(s)\mathbb{E}[X^*(s)] + D(s)w^*(s) + \bar{D}(s)\mathbb{E}[w^*(s)]+\sigma(s) \Big\}dW(s),\\
				&X^*(t) = \xi.
		\end{split}\right. \end{equation}
		We set \begin{equation}	\label{eq:4.18}
			\left\{	\begin{aligned}
				\mathbf{Y}^*(\cdot)=&\mathbf{P}(\cdot)(X^*(\cdot)-\mathbb{E}[X^*(\cdot)])+\mathbf{\Pi}(\cdot)\mathbb{E}[X^*(\cdot)]+\eta(\cdot)-\mathbb{E}[\eta(\cdot)]+\bar{\eta}(\cdot),\\
				\mathbf{Z}^*(\cdot)=&\mathbf{P}\bigg[\{A_2-D\mathbf{\Sigma}^{-1}\mathbf{J}(\mathbf{B}^\top\mathbf{P}+\mathbf{D}^\top\mathbf{P}A_2)\}(X^*-\mathbb{E}[X^*])\\&+\{\widetilde{A}_2-\widetilde{D}\mathbf{\bar{\Sigma}}^{-1}\mathbf{J}(\mathbf{\widetilde{B}}^\top\mathbf{\Pi}+\mathbf{\widetilde{D}}^\top\mathbf{P}\widetilde{A}_2)\}\mathbb{E}[X^*]\\&-D\mathbf{\Sigma}^{-1}\mathbf{J}^\top\{\mathbf{B}^{\top}(\eta-\mathbb{E}[\eta])+\mathbf{D}^\top(\zeta-\mathbb{E}[\zeta])+\mathbf{D}^\top\mathbf{P}(\sigma-\mathbb{E}[\sigma])\}\\&-\widetilde{D}\mathbf{\bar{\Sigma}}^{-1}\mathbf{J}^\top\{\mathbf{\widetilde{B}}^\top\bar{\eta}+\mathbf{\widetilde{D}}^\top\mathbb{E}[\zeta]+\mathbf{\widetilde{D}}^\top\mathbf{P}\mathbb{E}[\sigma]\}\bigg]+\zeta,
			\end{aligned}\right.
		\end{equation}
		For the sake of convenient expression, let \[	\begin{split}
			f(s)=&-\biggl\{\mathbf{A_1}(s)^\top\eta(s)-(\mathbf{P}(s)B(s)+\mathbf{A_2}(s)^\top\mathbf{P}(s)D(s))\mathbf{\Sigma}(s)^{-1}\mathbf{J}^\top\{\mathbf{B}(s)^\top\eta(s)\\&+\mathbf{D}(s)^\top[\zeta(s)+\mathbf{P}(s)\sigma(s)]\}+\mathbf{A_2}(s)^\top[\zeta(s)+\mathbf{P}(s)\sigma(s)]+\mathbf{P}(s)b(s)\biggr\},
		\end{split}\]
		then BSDE \eqref{eq:4.17*} can be rewritten as
		\begin{equation*}\left\{\begin{split}
				&d\eta(s)=-f(s)ds+\zeta(s)dW(s),\\&\eta(T)=0.\end{split}\right.
		\end{equation*}
		
		Applying It\^{o}'s formula to \(\mathbf{Y}^*(\cdot)\),
		we get \begin{equation}\label{eq:4.21}
			\begin{split}
				d\mathbf{Y}^*(\cdot)	=&	\{\dot{\mathbf{P}}(X^*-\mathbb{E}[X^*])+\mathbf{P}\mathbf{A_1}(X^*-\mathbb{E}[X^*])+\mathbf{P}B(w^*-\mathbb{E}[w^*])+\mathbf{P}(b-\mathbb{E}[b])\\&+\dot{\mathbf{\Pi}}\mathbb{E}[X^*]+\mathbf{\Pi}\widetilde{A}_1\mathbb{E}[X^*]+\mathbf{\Pi}\widetilde{B}\mathbb{E}[w]+\mathbf{\Pi}\mathbb{E}[b]+\dot{\bar{\eta}}-(f-\mathbb{E}[f])\}dt+\{\mathbf{P}A_2(X^*-\mathbb{E}[X^*])\\&+\mathbf{P}\widetilde{A}_2\mathbb{E}[X^*]~+\mathbf{P}D(w^*-\mathbb{E}[w^*])+\mathbf{P}\widetilde{D}\mathbb{E}[w^*]+\mathbf{P}\sigma+\zeta\}dW(s).
			\end{split}
		\end{equation}
		By conducting additional computations, one can get
		\begin{equation}\label{eq:4.25*}
			d\mathbf{Y}^*(\cdot)=-\Big\{ \mathbf{A_1}(s)^\top\mathbf{Y}^*(s) + \mathbf{\bar{A}_1}(s)^\top\mathbb{E}[\mathbf{Y}^*(s)]+ \mathbf{A_2}(s)^\top\mathbf{Z}^*(s) + \mathbf{\bar{A}_2}(s)^\top\mathbb{E}[\mathbf{Z}^*(s)] + \mathbf{Q}\mathbf{I}_n(s)X^*(s) \Big\}ds + \mathbf{Z}^*(s)dW(s).
		\end{equation}
		It's easy to see that \eqref{eq:4.25*} is the dual equation of the state equation \eqref{eq:MF-SDE}. 		
		
		Then, through further calculations we can easily get
		the stationarity condition
		\begin{equation}
			\mathbf{J}^\top \Big\{ \mathbf{B}(s)^\top\mathbf{Y}^*(s) + \mathbf{\widetilde{B}}(s)^\top\mathbb{E}[\mathbf{Y}^*(s)] + \mathbf{D}(s)^\top\mathbf{Z}^*(s) + \mathbf{\widetilde{D}}(s)^\top\mathbb{E}[\mathbf{Z}^*(s)] + \mathbf{R}\mathbf{I}_m(s)w^*(s) \Big\} = 0,
		\end{equation}
		
		If, in addition, condition (iv) holds, then by Theorem \ref{thm:4.3}, the open-loop  \(H_2/H_\infty\) problem admits an \emph{open-loop Nash equilibrium strategy} \(w^*(\cdot)\), which
		has the closed-loop representation \[w^*(\cdot) = K^*(\cdot)\left(X^*(\cdot) - \mathbb{E}[X^*(\cdot)]\right) + \widetilde{K}^*(\cdot)\mathbb{E}[X^*(\cdot)] + \chi^*(\cdot).\]
		The proof is complete.
	\end{proof}
	
	\begin{rmk}\label{rmk:4.7}
		This theorem indicates the sufficient conditions for the existence of a specific state-feedback form strategy for the open-loop \(H_2/H_\infty\) control problem, and provides a theoretical foundation for evaluating and constructing such control strategies.  Meanwhile, it should be noted that although this condition is sufficient, it is not necessary.  Since the open-loop strategy is related to a specific initial state \(\xi\), and it cannot guarantee that \(\mathbb{E}[X(s)]\) and \(X(s)-\mathbb{E}[X(s)]\) can take arbitrary values at all times \(s\in [t,T]\), thus affecting the solvability of CDREs.
	\end{rmk}

\subsection{Closed-Loop Strategy Analysis}

We investigate in this subsection the closed-loop Nash equilibrium for \textbf{Problem (NSDG-\(H_2/H_\infty\))}, which is formulated under mean-field stochastic dynamics and dual performance constraints. The equilibrium strategy is derived by solving a system of coupled Riccati differential equations along with backward and forward equations accounting for affine terms.

To facilitate the expression of these equations, we introduce the following symbols for representation:
\[
\begin{aligned}
& \varLambda := \gamma^2 I + C_2^\top P_1 C_2, \quad
\bar{\varLambda} := \gamma^2 I + \widetilde{C}_2^\top P_1 \widetilde{C}_2, \\
& \Theta := I + B_2^\top P_2 B_2, \quad
\bar{\Theta} := I + \widetilde{B}_2^\top P_2 \widetilde{B}_2, \\
& \Upsilon := P_1 C_1 + (A_2 + B_2 U)^\top P_1 C_2, \quad
\bar{\Upsilon} := \Pi_1 \widetilde{C}_1 + (\widetilde{A}_2 + \widetilde{B}_2 \widetilde{U})^\top P_1 \widetilde{C}_2, \\
& \Sigma := P_2 B_1 + (A_2 + C_2 V)^\top P_2 B_2, \quad
\bar{\Sigma} := \Pi_2 \widetilde{B}_1 + (\widetilde{A}_2 + \widetilde{C}_2 \widetilde{V})^\top P_2 \widetilde{B}_2, \\
& \Xi := C_2^\top \big[P_1 (B_2(U_0 - \mathbb{E}[U_0]) + \sigma - \mathbb{E}[\sigma]) + \zeta_1 - \mathbb{E}[\zeta_1] \big] + C_1^\top \eta_1, \\
& \bar{\Xi} := \widetilde{C}_2^\top \big[P_1(\widetilde{B}_2 \mathbb{E}[U_0] + \mathbb{E}[\sigma]) + \mathbb{E}[\zeta_1] \big] + \widetilde{C}_1^\top \bar{\eta}_1, \\
& \Psi := B_2^\top \big[P_2(C_2(V_0 - \mathbb{E}[V_0]) + \sigma - \mathbb{E}[\sigma]) + \zeta_2 - \mathbb{E}[\zeta_2] \big] + B_1^\top \eta_2, \\
& \bar{\Psi} := \widetilde{B}_2^\top \big[P_2(\widetilde{C}_2 \mathbb{E}[V_0] + \mathbb{E}[\sigma]) + \mathbb{E}[\zeta_2] \big] + \widetilde{B}_1^\top \bar{\eta}_2.
\end{aligned}
\]

The following two sets of cross-coupled Riccati equations are introduced before the main result is presented:
\begin{equation}\label{eq:5.3}
  \begin{cases}
      \dot{P}_1 + (A_1 + B_1 U)^\top P_1 + P_1 (A_1 + B_1 U) \\
        \qquad+ (A_2 + B_2 U)^\top P_1 (A_2 + B_2 U) - Q^\top Q - U^\top U - \Upsilon \varLambda^{-1} \Upsilon^\top = 0, \\
      \dot{\Pi}_1 + (\widetilde{A}_1 + \widetilde{B}_1 \widetilde{U})^\top \Pi_1 + \Pi_1 (\widetilde{A}_1 + \widetilde{B}_1 \widetilde{U}) \\
        \qquad+ (\widetilde{A}_2 + \widetilde{B}_2 \widetilde{U})^\top P_1 (\widetilde{A}_2 + \widetilde{B}_2 \widetilde{U}) - Q^\top Q - \widetilde{U}^\top \widetilde{U} - \bar{\Upsilon} \bar{\varLambda}^{-1} \bar{\Upsilon}^\top = 0, \\
      P_1(T) = 0, \quad \Pi_1(T) = 0, \quad \Lambda(s) \geq \delta_1 I,\quad \bar{\Lambda}(s) \geq \delta_1 I,\quad \text{for some }
      \delta_1 > 0 \text{ and } \forall s \in [t, T].
  \end{cases}
\end{equation}
	\begin{equation}\label{eq:5.4}
	\begin{cases}
		V=-\varLambda^{-1}\Upsilon^{\top},
		\\\bar{V}=-\bar{\varLambda}^{-1}\bar{\Upsilon}^\top+\varLambda^{-1}\Upsilon^{\top},\qquad\qquad\qquad\qquad\qquad\qquad\qquad\qquad\qquad\qquad\qquad\qquad\qquad\quad
	\end{cases}
\end{equation}
\begin{equation}\label{eq:5.5}
	\begin{cases}
		\dot{P}_2 + (A_1 + C_1 V)^\top P_2 + P_2 (A_1 + C_1 V) + (A_2 + C_2 V)^\top P_2 (A_2 + C_2 V) \\
		\qquad + Q^\top Q - \Sigma \Theta^{-1} \Sigma^\top = 0, \\
		\dot{\Pi}_2 + (\widetilde{A}_1 + \widetilde{C}_1 \widetilde{V})^\top \Pi_2 + \Pi_2 (\widetilde{A}_1 + \widetilde{C}_1 \widetilde{V}) + (\widetilde{A}_2 + \widetilde{C}_2 \widetilde{V})^\top P_2 (\widetilde{A}_2 + \widetilde{C}_2 \widetilde{V}) \\
		\qquad + Q^\top Q - \bar{\Sigma} \bar{\Theta}^{-1} \bar{\Sigma}^\top = 0, \\
		P_2(T) = 0, \quad \Pi_2(T) = 0, \quad \Theta(s) \geq \delta_2 I,\quad \bar{\Theta}(s) \geq \delta_2 I,\quad \text{for some } \delta_2 > 0 \text{ and } \forall s \in [t, T].
	\end{cases}
\end{equation}
\begin{equation}\label{eq:5.6}
	\begin{cases}
		U=-\Theta^{-1}\Sigma^{\top},
		\\\bar{U}=-\bar{\Theta}^{-1}\bar{\Sigma}^\top+\Theta^{-1}\Sigma^{\top}.\qquad\qquad\qquad\qquad\qquad\qquad\qquad\qquad\qquad\qquad\qquad\qquad\qquad\quad
	\end{cases}
\end{equation}

The affine correction terms for the disturbance and control strategies are characterized via the following BSDEs and ODEs.

The first BSDE determines the offset term \((\eta_1, \zeta_1)\) in the worst-case disturbance feedback:
\begin{equation}\label{eq:BSDE-eta1}
	\left\{
	\begin{aligned}
		d\eta_1 = & -\Big[ \big( A_{1} + B_{1}U \big)^{\top} - \Upsilon\varLambda^{-1}C_{1}^{\top} \Big] \eta_1 \\
		& - \Big[ \big( A_{2} + B_{2}U \big)^{\top} - \Upsilon\varLambda^{-1}C_{2}^{\top} \Big] \Big\{ P_{1} \big( B_{2}(U_0 - \mathbb{E}[U_0]) + \sigma - \mathbb{E}[\sigma] \big) + \zeta_1 - \mathbb{E}[\zeta_1] \Big\} \\
		& - P_{1} \big( B_{1}(U_0 - \mathbb{E}[U_0]) + b - \mathbb{E}[b] \big) \, ds + \zeta_1 \, dW(s), \\
		\eta_1(T) = & \ 0,
	\end{aligned}
	\right.
\end{equation}
\begin{equation}\label{eq:V0-definition}
{
V_0(s) = -\varLambda^{-1}(s) \Xi(s) - \bar{\varLambda}^{-1}(s) \bar{\Xi}(s)
}.
\end{equation}

The second BSDE characterizes the correction \((\eta_2, \zeta_2)\) for the optimal control:
\begin{equation}\label{eq:BSDE-eta2}
	\left\{
	\begin{aligned}
		d\eta_2 = & -\Big[ \big( A_{1} + C_{1}V \big)^{\top} - \Sigma\Theta^{-1}B_{1}^{\top} \Big] \eta_2 \\
		& - \Big[ \big( A_{2} + C_{2}V \big)^{\top} - \Sigma\Theta^{-1}B_{2}^{\top} \Big] \Big\{ P_{2} \big( C_{2}(V_0 - \mathbb{E}[V_0]) + \sigma - \mathbb{E}[\sigma] \big) + \zeta_2 - \mathbb{E}[\zeta_2] \Big\} \\
		& - P_{2} \big( C_{1}(V_0 - \mathbb{E}[V_0]) + b - \mathbb{E}[b] \big) \, ds + \zeta_2 \, dW(s), \\
		\eta_2(T) = & \ 0,
	\end{aligned}
	\right.
\end{equation}

\begin{equation}\label{eq:U0-definition}
{
U_0(s) = -\Theta^{-1}(s) \Psi(s) - \bar{\Theta}^{-1}(s) \bar{\Psi}(s)
}.
\end{equation}

The mean-field corrections \(\bar{\eta}_1\) and \(\bar{\eta}_2\) satisfy the following deterministic ODEs.

The first one is associated with the disturbance feedback:
\begin{equation}\label{eq:ode-bar-eta1}
	\left\{
	\begin{aligned}
		\dot{\bar{\eta}}_{1} + & \Big[ (\widetilde{A}_{1} + \widetilde{B}_{1}\widetilde{U})^{\top} - \bar{\Upsilon}\bar{\varLambda}^{-1}\widetilde{C}_{1}^{\top} \Big] \bar{\eta}_{1} \\
		& + \Big[ (\widetilde{A}_{2} + \widetilde{B}_{2}\widetilde{U})^{\top} - \bar{\Upsilon}\bar{\varLambda}^{-1}\widetilde{C}_{2}^{\top} \Big] \Big\{ P_{1} \big( \widetilde{B}_{2}\mathbb{E}[U_0] + \mathbb{E}[\sigma] \big) + \mathbb{E}[\zeta_1] \Big\} \\
		& + \Pi_1 \big( \mathbb{E}[b] + \widetilde{B}_{1}\mathbb{E}[U_0] \big) = 0, \\
		\bar{\eta}_{1}(T) &=  \ 0.
	\end{aligned}
	\right.
\end{equation}

The second governs the mean-field correction of the optimal control:
\begin{equation}\label{eq:ode-bar-eta2}
\left\{
\begin{aligned}
	\dot{\bar{\eta}}_{2} + & \Big[ (\widetilde{A}_{1} + \widetilde{C}_{1}\widetilde{V})^{\top} - \bar{\Sigma}\bar{\Theta}^{-1}\widetilde{B}_{1}^{\top} \Big] \bar{\eta}_{2} \\
	& + \Big[ (\widetilde{A}_{2} + \widetilde{C}_{2}\widetilde{V})^{\top} - \bar{\Sigma}\bar{\Theta}^{-1}\widetilde{B}_{2}^{\top} \Big] \Big\{ P_{2} \big( \widetilde{C}_{2}\mathbb{E}[V_0] + \mathbb{E}[\sigma] \big) + \mathbb{E}[\zeta_2] \Big\} \\
	& + \Pi_2 \big( \mathbb{E}[\sigma] + \widetilde{C}_{1}\mathbb{E}[V_0] \big) = 0, \\
	\bar{\eta}_{2}(T) &=  \ 0.
\end{aligned}
\right.
\end{equation}

\begin{lem}\label{lem:5.1}  Let Assumptions \ref{ass:3.1}-\ref{ass:3.3} holds.
Let \((V(\cdot), \bar{V}(\cdot), V_0(\cdot)) \in \mathcal{M}^2[t, T]\) be a given admissible disturbance strategy, and suppose the disturbance process takes the affine feedback form
\begin{equation}\label{eq:disturbance-form}
v(s) = V(s)X(s) + \bar{V}(s)\mathbb{E}[X(s)] + V_0(s), \quad s \in [t, T].
\end{equation}
Substituting \eqref{eq:disturbance-form} into the state equation \eqref{eq:2.1}, the resulting mean-field LQ optimal control problem with affine terms becomes
\begin{equation}\label{eq:4.22}
	\begin{cases}
	\underset{u(\cdot)\in L^{2}_{\mathbb {F}}(t,T;\mathbb{R}^{n_{u}})}{\min}J_{2}(t,\xi;u(\cdot),v(\cdot))=\underset{u(\cdot)\in L^{2}_{\mathbb {F}}(t,T;\mathbb{R}^{n_{u}})}{\min}{\mathbb{E}\int_{t}^{T}[\|Q(s)X(s)\|^2+\|u(s)\|^2]ds},\\
	dX(s)=\Big\{(A_{1}(s)+C_{1}(s)V(s))X(s)+(\bar{A}_{1}(s)+C_1(s)\bar{V}(s)+\bar{C}_{1}(s)\widetilde{V}(s))\mathbb{E}[X(s)] \\
	\qquad\qquad+C_1(s)V_0(s)+\bar{C}_1(s)\mathbb{E}[V_0(s)]+B_{1}(s)u(s)+\bar{B}_{1}(s)\mathbb{E}[u(s)]+b(s)\Big\}ds\\
	\qquad\qquad+\Big\{(A_{2}(s)+C_{2}V(s))X(s)+(\bar{A}_{2}(s)+C_2(s)\bar{V}(s)+\bar{C}_{2}(s)\widetilde{V}(s))\mathbb{E}[X(s)] \\
	\qquad\qquad+C_2(s)V_0(s)+\bar{C}_{2}(s)\mathbb{E}[V_0(s)]+B_{2}(s)u(s)+\bar{B}_{2}(s)\mathbb{E}[u(s)]+\sigma(s)\Big\}dW(s),\\
	X(t)=\xi\in L^2(\Omega,\mathbb P,\mathcal {F}_t;\mathbb{R}^{n}),~~s\in[t,T].
\end{cases}
\end{equation}
Then the fully coupled Riccati equations \eqref{eq:5.5}--\eqref{eq:5.6} admit solutions
\[
(P_2(\cdot), \Pi_2(\cdot); U(\cdot), \bar{U}(\cdot), U_0(\cdot)) \in C([t,T]; \mathbb{S}^n)^2 \times \mathcal{N}^2[t, T],
\]
and the backward and forward equations admit unique solutions
\[
(\eta_2(\cdot), \zeta_2(\cdot)) \in L^2_{\mathbb{F}}(\Omega; C(t, T; \mathbb{R}^n)) \times L^2_{\mathbb{F}}(t, T; \mathbb{R}^n), \quad \bar{\eta}_2(\cdot) \in L^2(t, T; \mathbb{R}^n).
\]
 And  the mean-field stochastic LQ control problem \eqref{eq:4.22}  exists an unique  optimal control \(u^*(\cdot)\)  given by the feedback form:
\[
\begin{aligned}
u^*(s) =  -\Theta^{-1} \Sigma^\top X^*(s) + \left[ -\bar{\Theta}^{-1} \bar{\Sigma}^\top + \Theta^{-1} \Sigma^\top \right] \mathbb{E}[X^*(s)] - \Theta^{-1} \Psi(s) - \bar{\Theta}^{-1} \bar{\Psi}(s),
\end{aligned}
\]
where \(X^*(\cdot)\) is the optimal state trajectory under \(u^*(\cdot)\).

The corresponding optimal cost is given by
\begin{equation*}
\begin{aligned}
J_2(t, \xi; u^*(\cdot), v(\cdot) = &~ \mathbb{E} \langle P_2(t)(\xi - \mathbb{E}[\xi]) + 2\eta_2(t), \xi - \mathbb{E}[\xi] \rangle + \langle \Pi_2(t)\mathbb{E}[\xi], \mathbb{E}[\xi] \rangle + 2 \langle \bar{\eta}_2(t), \mathbb{E}[\xi] \rangle \\
& + \mathbb{E} \int_t^T \bigg[ \langle P_2(\sigma + C_2 V + \bar{C}_2 \mathbb{E}[V_0]), \sigma + C_2 V + \bar{C}_2 \mathbb{E}[V_0] \rangle \\
&\quad + 2 \langle \eta_2, b - \mathbb{E}[b] + C_1(V_0 - \mathbb{E}[V_0]) \rangle + 2 \langle \zeta_2, \sigma + C_2 V + \bar{C}_2 \mathbb{E}[V_0] \rangle \\
&\quad + 2 \langle \bar{\eta}_2, \widetilde{C}_1 \mathbb{E}[V_0] + \mathbb{E}[b] \rangle -  \langle \Theta^{-1} \Psi, \Psi\rangle - \langle \bar{\Theta}^{-1} \bar{\Psi}, \bar{\Psi} \rangle \bigg] ds.
\end{aligned}
\end{equation*}
\end{lem}

The proof is omitted here and can be found in \cite{sun2016open}.

\begin{thm}\label{thm:4.9-final}   Let Assumptions \ref{ass:3.1}-\ref{ass:3.3} holds. And
suppose the following conditions hold:
\begin{itemize}
\item[(i)] The system of coupled Riccati differential equations \eqref{eq:5.3}--\eqref{eq:5.6} admits a solution quadruple
\[
(P_1(\cdot), P_2(\cdot), \Pi_1(\cdot), \Pi_2(\cdot)) \in C([t, T]; \mathbb{S}^n)^4,
\]
with corresponding feedback gains
\[
(U(\cdot), \bar{U}(\cdot), U_0(\cdot)) \in \mathcal{N}^2[t, T], \quad (V(\cdot), \bar{V}(\cdot), V_0(\cdot)) \in \mathcal{M}^2[t, T].
\]

\item[(ii)] For each \( i = 1, 2 \), the BSDEs \eqref{eq:BSDE-eta1} and \eqref{eq:BSDE-eta2} admit unique adapted solutions
\[
(\eta_i(\cdot), \zeta_i(\cdot)) \in L^2_{\mathbb{F}}(\Omega; C([t, T]; \mathbb{R}^n)) \times L^2_{\mathbb{F}}(t, T; \mathbb{R}^n),
\]
and the deterministic ODEs \eqref{eq:ode-bar-eta1}--\eqref{eq:ode-bar-eta2} admit unique solutions \(\bar{\eta}_i(\cdot) \in L^2(t, T; \mathbb{R}^n)\).
\end{itemize}

Then, the strategy pair
\[
\left(U(\cdot), \bar{U}(\cdot), U_0(\cdot); V(\cdot), \bar{V}(\cdot), V_0(\cdot)\right)
\]
constitutes
\begin{enumerate}
\item \textbf{A closed-loop Nash equilibrium for Problem (NSDG-\(H_2/H_\infty\))} (as defined in Problem~\ref{def:NSDG}), in the sense that:
\begin{align}
J_1(t,\xi; u^*(\cdot), v^*(\cdot)) &\leq J_1\left(t,\xi; U(\cdot)X(\cdot) + \bar{U}(\cdot)\mathbb{E}[X(\cdot)] + U_0(\cdot), v(\cdot)\right),  \quad \forall v(\cdot) \in L^2_{\mathbb{F}}(t,T;\mathbb{R}^{n_v}), \label{eq:nash1-final} \\
J_2(t,\xi; u^*(\cdot), v^*(\cdot)) &\leq J_2\left(t,\xi; u(\cdot), V(\cdot)X(\cdot) + \bar{V}(\cdot)\mathbb{E}[X(\cdot)] + V_0(\cdot)\right),  \quad \forall u(\cdot) \in L^2_{\mathbb{F}}(t,T;\mathbb{R}^{n_u}). \label{eq:nash2-final}
\end{align}

\item \textbf{An Optimal Closed-Loop Strategy to the mean-field stochastic \(H_2/H_\infty\) control problem} (as defined in Definition~\ref{def:1}), satisfying:\\
(i) The closed-loop system achieves disturbance attenuation with performance level \(\gamma\):
\[
\|\mathbf{L}\| < \gamma.
\]
(ii) The optimal \(H_2\) performance criterion is attained:
\begin{equation}\label{eq:H2-optimality}
\begin{aligned}
J_2(t, \xi; u^*(\cdot), v^*(\cdot)) &\leq J_2\left(t, \xi; u(\cdot), V(\cdot)X(\cdot) + \bar{V}(\cdot)\mathbb{E}[X(\cdot)] + V_0(\cdot)\right), \quad \forall u(\cdot) \in L^2_{\mathbb{F}}(t, T; \mathbb{R}^{n_u}).\\
\end{aligned}
\end{equation}
\end{enumerate}
Here, the optimal closed-loop strategies \(u^*(\cdot)\) and \(v^*(\cdot)\) are given by:
\begin{align}
u^*(s) &= U(s)X^*(s)+  \bar{U}(s)\mathbb{E}[X^*(s)] + U_0(s), \label{eq:equil_u-final} \\
v^*(s) &= V(s)X^*(s)  +  \bar{V}(s)\mathbb{E}[X^*(s)] + V_0(s), \label{eq:equil_v-final}
\end{align}
where \(X^*(\cdot)\) denotes the corresponding state trajectory under the feedback strategies \(u^*(\cdot)\) and \(v^*(\cdot)\).
\end{thm}

\begin{proof}
Suppose the coupled Riccati equations \eqref{eq:5.3}--\eqref{eq:5.6} admit a unique quadruple solution  \((P_1, P_2, \Pi_1, \Pi_2)\), with associated gains \((U, \bar{U})\), \((V, \bar{V})\), and the BSDEs \eqref{eq:BSDE-eta1}--\eqref{eq:U0-definition}, ODEs \eqref{eq:ode-bar-eta1}--\eqref{eq:ode-bar-eta2} admit solutions \((\eta_i, \zeta_i)\), \(\bar{\eta}_i\) with \(U_0 \) and \(V_0\) for \(i = 1, 2\).

Substituting
\[
u(s) = U(s)X(s) + \bar{U}(s)\mathbb{E}[X(s)] + U_0(s),
\]
into the state equation \eqref{eq:2.1} yields the following closed-loop dynamics:
\begin{equation}\label{eq:closed_loop_system}
	\left\{
	\begin{aligned}
		dX(s) = & \Bigl[ (A_{1}(s) + B_{1}(s)U(s)) X(s) + (\bar{A}_{1}(s) + B_1(s)\bar{U}(s) + \bar{B}_{1}(s)\widetilde{U}(s))\mathbb{E}[X(s)] \\
		& \quad + B_1(s)U_0(s) + \bar{B}_1(s)\mathbb{E}[U_0(s)] + C_{1}(s)v(s) + \bar{C}_{1}(s)\mathbb{E}[v(s)] + b(s) \Bigr] ds \\
		& + \Bigl[ (A_{2}(s) + B_{2}(s)U(s)) X(s) + (\bar{A}_{2}(s) + B_2(s)\bar{U}(s) + \bar{B}_{2}(s)\widetilde{U}(s))\mathbb{E}[X(s)] \\
		& \quad +B_2(s)U_0(s)+ \bar{B}_{2}(s)\mathbb{E}[U_0(s)] + C_{2}(s)v(s) + \bar{C}_{2}(s)\mathbb{E}[v(s)] + \sigma(s) \Bigr] dW(s), \quad s \in [t,T], \\
		X(t) = & \ \xi,
	\end{aligned}
	\right.
\end{equation}

 Let \(\bar{X}(s) := X(s) - \mathbb{E}[X(s)]\), \(\bar{Z}(s) := v(s) - \mathbb{E}[v(s)]\), and apply It\^{o}'s formula to the process \(\langle P_1\bar{X} + 2\eta_1, \bar{X} \rangle\) and \(\langle \Pi_1\mathbb{E}[X] + 2\bar{\eta}_1, \mathbb{E}[X] \rangle\). After a detailed computation , we obtain
\begin{equation}\label
{eq:proof-J1-equality}
\begin{aligned}
& J_1(t, \xi; u(\cdot), v(\cdot)) - \mathbb{E}\langle P_1(\xi - \mathbb{E}[\xi]) + 2\eta_1, \xi - \mathbb{E}[\xi] \rangle - \langle \Pi_1(t)\mathbb{E}[\xi], \mathbb{E}[\xi] \rangle - 2\langle \bar{\eta}_1(t), \mathbb{E}[\xi] \rangle \\
=~
& \mathbb{E} \int_t^T \Big
[
\langle \varLambda(\bar{Z} + \varLambda^{-1}\Upsilon^\top\bar{X} + \varLambda^{-1}\Xi), \bar{Z} + \varLambda^{-1}\Upsilon^\top\bar{X} + \varLambda^{-1}\Xi \rangle - \langle \varLambda^{-1}\Xi, \Xi \rangle \\
&\quad + \langle \bar{\varLambda}(\mathbb{E}[v] + \bar{\varLambda}^{-1}\bar{\Upsilon}^\top\mathbb{E}[X] + \bar{\varLambda}^{-1}\bar{\Xi}), \mathbb{E}[v] + \bar{\varLambda}^{-1}\bar{\Upsilon}^\top\mathbb{E}[X] + \bar{\varLambda}^{-1}\bar{\Xi} \rangle - \langle \bar{\varLambda}^{-1}\bar{\Xi}, \bar{\Xi} \rangle \\
&\quad + \langle P_1(\sigma + B_2 U_0 + \bar{B}_2 \mathbb{E}[U_0]), \sigma + B_2 U_0 + \bar{B}_2 \mathbb{E}[U_0] \rangle \\
&\quad + 2\langle \eta_1, b - \mathbb{E}[b] + B_1(U_0 - \mathbb{E}[U_0]) \rangle + 2\langle \zeta_1, \sigma + B_2 U_0 + \bar{B}_2 \mathbb{E}[U_0] \rangle \\
&\quad + 2\langle \bar{\eta}_1, \mathbb{E}[b] + \widetilde{B}_1 \mathbb{E}[U_0] \rangle
\Big
] ds.
\end{aligned}
\end{equation}

Let \(X^*(\cdot)\) be the state trajectory induced by the control pair \((u^*(\cdot), v^*(\cdot))\). By defining the closed-loop strategy as
\begin{equation}\label{eq:4.39}
	\begin{aligned}
		u^*(\cdot) &= U(\cdot)X^*(\cdot) + \bar{U}(\cdot)\mathbb{E}[X^*(\cdot)] + U_0(\cdot), \\
		v^*(\cdot) &= V(\cdot)X^*(\cdot) + \bar{V}(\cdot)\mathbb{E}[X^*(\cdot)] + V_0(\cdot),
	\end{aligned}
\end{equation}
we observe that \( \bar{Z}^*(s) = -\varLambda^{-1} \Upsilon^\top \bar{X}(s) - \varLambda^{-1}\Xi \), and \( \mathbb{E}[v^*(s)] = -\bar{\varLambda}^{-1} \bar{\Upsilon}^\top \mathbb{E}[X(s)] - \bar{\varLambda}^{-1} \bar{\Xi} \). Substituting these into \eqref{eq:proof-J1-equality}, the integral terms vanish, yielding
\begin{equation}
	\begin{aligned}
		J_1(t, \xi; u^*(\cdot), v^*(\cdot))
		=~ & \mathbb{E} \langle P_1(\xi - \mathbb{E}[\xi]) + 2\eta_1, \xi - \mathbb{E}[\xi] \rangle + \langle \Pi_1(t)\mathbb{E}[\xi], \mathbb{E}[\xi] \rangle \\
		& + 2\langle \bar{\eta}_1(t), \mathbb{E}[\xi] \rangle \\
		& + \mathbb{E} \int_t^T \Big[
		\langle P_1(\sigma + B_2 U_0 + \bar{B}_2 \mathbb{E}[U_0]), \sigma + B_2 U_0 + \bar{B}_2 \mathbb{E}[U_0] \rangle \\
		&\quad + 2\langle \eta_1, b - \mathbb{E}[b] + B_1(U_0 - \mathbb{E}[U_0]) \rangle + 2\langle \zeta_1, \sigma + B_2 U_0 + \bar{B}_2 \mathbb{E}[U_0] \rangle \\
		&\quad + 2\langle \bar{\eta}_1, \mathbb{E}[b] + \widetilde{B}_1 \mathbb{E}[U_0] \rangle - \langle \varLambda^{-1} \Xi, \Xi \rangle - \langle \bar{\varLambda}^{-1} \bar{\Xi}, \bar{\Xi} \rangle
		\Big] ds.
	\end{aligned}
\end{equation}

Using the nonnegativity of the squared terms in \eqref{eq:proof-J1-equality}, it follows that for any \(v(\cdot) \in L^2_{\mathbb{F}}(t, T; \mathbb{R}^{n_v})\),
\[
J_1(t, \xi; u^*(\cdot), v^*(\cdot)) \leq J_1(t, \xi;U(\cdot)X(\cdot)+\bar{U}(\cdot)\mathbb{E}[X(\cdot)]+U_0(\cdot), v(\cdot)).
\]
  In fact, substituting $v(\cdot)=V(\cdot)X(\cdot)+\bar{V}(\cdot)\mathbb{E}[X(\cdot)]+V_0(\cdot)$  in system \eqref{eq:2.1}, we get \eqref{eq:2.4}. Then minimizing \(J_2(t,\xi;u(\cdot),v(\cdot))\) is a standard stochastic LQ optimization problem. Considering Lemma \ref{lem:5.1} and closed-loop stategy \eqref{eq:4.39}, we obtain
\[
J_2(t, \xi; u^*(\cdot), v^*(\cdot)) \leq J_2(t, \xi; u(\cdot), V(\cdot)X(\cdot) + \bar{V}(\cdot)\mathbb{E}[X(\cdot)] + V_0(\cdot)),\quad \forall u(\cdot) \in L^2_{\mathbb{F}}(t, T; \mathbb{R}^{n_u}).
\]

Furthermore, the optimal control input \[
\begin{aligned}
	u^*(s) = & \ -\Theta^{-1}\Sigma^{\top}X^*(s) + \biggl\{ -\bar{\Theta}^{-1}\bar{\Sigma}^\top + \Theta^{-1}\Sigma^{\top} \biggr\} \mathbb{E}[X^*(s)]  - \Theta^{-1}\Psi- \bar{\Theta}^{-1}\bar{\Psi},\\=&U(s)X^*(s) + \bar{U}(s)\mathbb{E}[X^*(s)] + U_0(s)
\end{aligned}  \]
can be obtained. Finally, by Lemma~\ref{lem:3.3}, when \(\xi = 0\), \(b(\cdot) = \sigma(\cdot) = 0\), we obtain \(\|\mathbf{L}\| < \gamma\), verifying the \(H_\infty\) attenuation.
\end{proof}

\begin{rmk}
Theorem \ref{thm:4.9-final} indicates that for the \textbf{problem (NSDG-$H_2/H_\infty$)}, the sufficient condition for the existence of the closed-loop Nash equilibrium strategy is the solvability of the fully coupled Riccati equation  \eqref{eq:5.3}--\eqref{eq:5.6}, the BSDEs \eqref{eq:BSDE-eta1}--\eqref{eq:BSDE-eta2}, and the ODEs \eqref{eq:ode-bar-eta1}--\eqref{eq:ode-bar-eta2}. Moreover, as discussed in Remark \ref{rmk:4.7}, in the mean-field setting with affine terms, the solvability of the open-loop does not imply the existence of the closed-loop $H_2/H_\infty$ strategy.
\end{rmk}

	\section{Conclusion}	
In this paper, we studied the mixed $H_2/H_\infty$ control problem for continuous-time mean-field stochastic systems with affine terms over a finite horizon. By establishing a unified framework that integrates mean-field theory and stochastic differential equations, we derived necessary and sufficient conditions for system robustness via the MF-SBRL and characterized both open-loop and closed-loop solvability through coupled Riccati and auxiliary equations. The results reveal a linear feedback structure involving both the state and its mean, thereby extending classical control theory to more complex mean-field settings. This work lays the groundwork for further studies on infinite-horizon problems, nonlinear dynamics, decentralized control, and data-driven methods in large-scale stochastic systems.

\end{document}